\documentclass[reqno]{amsart}
\usepackage{amsmath,amssymb,amsthm, mathrsfs} 
\usepackage{hyperref} 
\usepackage{cleveref}
\usepackage{autonum}
\usepackage{enumerate}
\usepackage{color}
\usepackage[all]{xy}
\usepackage{fullpage}
\usepackage{todonotes}
\usepackage{mathptmx}
\usepackage{arydshln}

\newtheorem{theorem}{Theorem}
\newtheorem{thm}[theorem]{Theorem}
\newtheorem{prop}[theorem]{Proposition}
\newtheorem{lemma}[theorem]{Lemma}

\theoremstyle{definition}
\newtheorem{defn}[theorem]{Definition}
\newtheorem{remark}[theorem]{Remark}
\newtheorem{example}[theorem]{Example}
\newtheorem{question}[theorem]{Question}

\newtheorem*{thmA}{Theorem A}
\newtheorem*{thmB}{Theorem B}
\newtheorem*{thmC}{Theorem C}

\Crefname{thm}{Theorem}{Theorems}
\Crefname{prop}{Proposition}{Propositions}
\Crefname{lemma}{Lemma}{Lemmas}
\Crefname{cor}{Corollary}{Corollaries}
\Crefname{defn}{Definition}{Definitions}
\Crefname{remark}{Remark}{Remarks}
\Crefname{example}{Example}{Examples}
\Crefname{question}{Question}{Questions}
\Crefname{assum}{Assumption}{Assumptions}

\DeclareMathOperator{\Hom}{Hom}

\DeclareMathOperator{\rank}{rank}

\DeclareMathOperator{\Sym}{Sym}

\DeclareMathOperator{\ann}{ann}

\DeclareMathOperator{\hgt}{ht}

\DeclareMathOperator{\rk}{rk}
\DeclareMathOperator{\spc}{Spec}

\DeclareMathOperator{\coker}{coker}
\DeclareMathOperator{\ima}{Im}
\DeclareMathOperator{\Proj}{Proj}
\DeclareMathOperator{\Spec}{Spec}
\DeclareMathOperator{\Quot}{Quot}
\DeclareMathOperator{\depth}{depth}
\DeclareMathOperator{\grade}{grade}
\DeclareMathOperator{\ide}{id}
\DeclareMathOperator{\fitt}{Fitt}

\newcommand{\p}{\mathfrak{p}}

\newcommand{\m}{\mathfrak{m}}

\newcommand{\kk}{{\mathbf k}}
\newcommand{\PP}{{\mathbb P}_\kk}
\newcommand{\nn}{{\mathbf q}}
\newcommand{\iffw}{\text{if and only if}}

\makeatletter
\newcommand{\xleftrightarrow}[2][]{\ext@arrow 3359\leftrightarrowfill@{#1}{#2}}
\newcommand{\xdashrightarrow}[2][]{\ext@arrow 0359\rightarrowfill@@{#1}{#2}}
\newcommand{\xdashleftarrow}[2][]{\ext@arrow 3095\leftarrowfill@@{#1}{#2}}
\newcommand{\xdashleftrightarrow}[2][]{\ext@arrow 3359\leftrightarrowfill@@{#1}{#2}}
\def\rightarrowfill@@{\arrowfill@@\relax\relbar\rightarrow}
\def\leftarrowfill@@{\arrowfill@@\leftarrow\relbar\relax}
\def\leftrightarrowfill@@{\arrowfill@@\leftarrow\relbar\rightarrow}
\def\arrowfill@@#1#2#3#4{%
  $\m@th\thickmuskip0mu\medmuskip\thickmuskip\thinmuskip\thickmuskip
   \relax#4#1
   \xleaders\hbox{$#4#2$}\hfill
   #3$%
}
\makeatother

\newcommand{\vertm}{\begin{matrix}1\\2\\3\\4\\5\\6\\7\\8\\9 \end{matrix}}
\newcommand{\vertmm}{\begin{matrix}1\\2\\3\\4\\5\\6\\7\\8\\9\\10 \end{matrix}}
\newcommand{\vphnm}{\vphantom{\vertm}}
\newcommand{\vphnmm}{\vphantom{\vertmm}}

\author{Youngsu Kim}
\address{Department of Mathematical Sciences, University of Arkansas, Fayetteville, Arkansas 72701  U.S.A}
\email{yk009@uark.edu}

\author{Vivek Mukundan}
\address{Department of Mathematics, University of Virginia, Charlottesville, Virginia 22902 U.S.A}
\email{vm6y@virginia.edu}


 \date
 \today

\title{Equations defining certain graphs}
\begin{document}
\maketitle

\begin{center}
{\textit{Dedicated to Professor Bernd Ulrich on the occasion of his 65th birthday}}
\end{center}

\begin{abstract}
Consider the rational map 	$\phi: \mathbb{P}^{n-1}_\kk {\xdashrightarrow{[f_0:\cdots: f_n]}}  \mathbb{P}^{n}_\kk$  defined by homogeneous polynomials $f_0,\dots,f_n$ of the same degree $d$ in a polynomial ring $R=\kk[x_1,\dots,x_n]$ over a field $\kk$. Suppose $I=(f_0,\dots,f_n)$ is a height two perfect ideal satisfying $\mu(I_p)\leq\dim R_p$ for $p\in \spc (R) \setminus V(x_1,\dots, x_n)$. 
We study the equations defining the graph of $\phi$ whose coordinate ring is the Rees algebra $R[It]$. 
We provide new methods to construct these equations using work of Buchsbaum and Eisenbud. 
Furthermore, for certain classes of ideals satisfying the conditions above, 
our methods lead to explicit equations defining Rees algebras of the ideals in these classes. 
These classes of examples are interesting, in that, there are no known methods to compute the defining ideal of the Rees algebra of such ideals.
These new methods also give rise to effective criteria to check that $\phi$ is birational onto its image.
\end{abstract}

\section{Introduction}
Our primary goal in this paper is to understand certain rational maps from the projective $n-1$ space to the projective $n$ space. 
In particular, we will provide an explicit description of the equations defining the image and graph of certain rational maps.
Fix a field $\kk$. Let $X = \mathbb{P}^{n-1}_\kk$ and $W$ a linear system in $
\operatorname{H}^0 (X, \mathcal{O}_X (d))$ for some positive integer $d$.
In general, $W$ does not induce a morphism. However, it defines a rational map from $X$ to $Y := \mathbb{P}^{\dim W - 1}_\kk$ with the base locus defined by $W$. We denote such a rational map by $\phi: X \stackrel{W}{\dashrightarrow} Y$. 
The image of $\phi$ may not be closed in $Y$. To study the algebraic properties of the image, we take the closure of the image $\phi(X)$ in $Y$. The questions we are interested in are the following:
\begin{question}\label{intQuestion}
\begin{enumerate}[\indent$(a)$]
\item What are the equations defining $\overline{\phi(X)}$ in $Y$. More generally, let $\Gamma_\phi$ denotes the closure of the graph of $\phi$. 
What are the equations defining $\Gamma_\phi$ in $X \times Y$?
\item Is the rational map $\phi$ birational onto its image?
\end{enumerate}
\end{question}

These questions have a tight connection with commutative algebra as their coordinate rings are well-studied rings among algebraists. 
In the coordinate ring $R = \kk[x_1,\dots, x_n]$ of $X = \mathbb{P}^{n-1}_\kk$, let $I$ be the ideal generated by the elements in $W$. The ideal $I$ is a homogeneous ideal generated in degree $d$. The \textit{Rees algebra}
\begin{equation} 
R[It] := R \oplus I t \oplus I^2 t^2 \oplus \cdots
\end{equation}
and the \textit{special fiber ring} 
\begin{equation} 
\mathcal{F}(I) := R[It] \otimes \kk \cong R/ \m \cong R[It]/\m R[It],
\end{equation}
where $\m$ is the maximal ideal $(x_1,\dots, x_n)$ are the coordinate rings of the projective varieties $\Gamma_\phi$ in $X \times Y$ and $\overline{\phi(X)}$ in $Y$, respectively.
Hence, the answer to \Cref{intQuestion}(a) is nothing but the equations in the defining ideals of these rings (see \Cref{setupSec} for details). 
In the case where ideal $I$ is \textit{nice}, for instance,  
if $W = \mathcal{O}_X (d)$ the d-th Veronese embedding or $I$ is a complete intersection ideal, the defining ideal of the Rees algebra of $I$ is well-understood, cf.\ \cite[Sec. 5.5]{SH}. 
However, determining defining ideals of Rees algebras is a challenging task, and there is a series of work on this topic. \\

Our main case concerns perfect ideals of codimension (equivalently height) two. Such ideals satisfy the structure theorem of Burch, and a number of papers has been devoted to understanding this class of ideals; for instance \cite{HSV83,Mor96,MU96,Ha02,HSV08,CHW08,Bus09,KPU11,CDA13,CDA14,Lan14,Mad15,BM16,KPU17}.
The novelty in our approach is to incorporate the work of Buchsbaum and Eisenbud \cite{BE74}, in particular, Buchsbaum-Eisenbud multipliers. \\

We work on \Cref{intQuestion} in the following setup. Let $W \subset \operatorname{H}^0(\mathbb{P}^{n-1}, \mathcal{O}_{\mathbb{P}^{n-1}} (d))$ be a subsystem. 
Assume that $\dim_{\kk} W = n+1$ (so $\phi: \mathbb{P}^{n-1}_\kk \stackrel{W}{\dashrightarrow} \mathbb{P}^{n}_\kk$), 
$\operatorname{codim} \operatorname{Fitt}_i (I) \ge i+1$ for $1 \le i \le n-1$, and $\Proj R/I$ is arithmetically Cohen-Macaulay of codimension $2$. Here, $\operatorname{Fitt}_i (I)$ denotes the $i$th Fitting ideal of $I$. 
The condition that $\Proj R/I$ is arithmetically Cohen-Macaulay of codimension 2 is equivalent to the condition of the base locus being defined by a height $2$ perfect ideal $I$.
The \textit{defining ideal} of the Rees algebra $R[It]$ is the kernel of the surjective $R$-algebra homomorphism
\begin{equation}\label{introPresentationB}
B := R \otimes_{\kk} S = R[T_0,\dots, T_n] \to R[It],
\end{equation}
where $S = \kk[T_0,\dots, T_n]$. 
This kernel a bigraded ideal in $B$. 
We often study the defining ideal of $R[It]$ in $\Sym_R(I)$, the symmetric algebra of $I$, since the presentation in \cref{introPresentationB} factors through $\Sym_R(I)$, and the kernel of this induced presentation is well-understood, see \Cref{setupSec}.
By abuse of notation, we often call the kernel of the induced map $\Sym_R(I) \to R[It]$ the \textit{defining ideal} of the Rees algebra $R[It]$. \\

Under this setup, $\Sym_R(I)$ is a complete intersection ring, and the defining ideal of the Rees algebra $R[It]$ in $\Sym_R(I)$ is equal to $\mathcal{A} := H^0_\m (\Sym_R(I))$, the zeroth local cohomology module supported in the maximal ideal $\m = (x_1, \dots, x_n)R$. 
These facts allow us to take the advantage of the bigraded structure of $\mathcal{A}$ and the duality theorem (\Cref{introDualThm}).  
In the sequel, for a bigraded module $M = \oplus_{i,j \in \mathbb{Z}} M_{i,j}$ in $B$, we set $M_i := \oplus_{j \in \mathbb{Z}} M_{i,j}$.
\begin{thm}[{Grothendieck, \cite[Section 3.6]{Jou97}, \cite[Theorem 2.4]{KPU17} }]\label{introDualThm}
For each $i \in \{ 0,\dots, d-n \}$, there exists an isomorphism of finitely generated graded $S$-modules
\begin{equation}
\mathcal{A}_i\cong \Hom_S (\Sym(I)_{d-n-i},S(-n)).
\end{equation}
\end{thm}
In \cite{Jou97}, Jouanolou constructed the isomorphisms in \Cref{introDualThm} explicitly. 
Hence to understand $\mathcal{A}_i$ it suffices to study its dual $\Hom_S (\Sym(I)_{d-n-i},S(-n))$. 
Our first result concerns this Hom module.
Fix $i \in \{0,\dots, d-n \}$ and let 
\begin{equation}\label{intPres}
F_1 \stackrel{\alpha}{\to} F_0 \to [\Sym_R(I)]_{d-n-i}
\end{equation}
be a graded presentation of $[\Sym_R(I)]_{d-n-i}$. 
Then $\Hom_S (\Sym(I)_{d-n-i}, S(-n))$ corresponds to the kernel of  $\alpha^*$. 
In {\Cref{thm1}}, we show that there exists a complex which induces elements in $\ker \alpha^*$.

\begin{thmA}[{\Cref{thm1}}]
We have a complex of graded free $S$-modules
\begin{equation}\label{eqIntro}
\wedge^{r-1}F_0\otimes S(-s_1)\stackrel{\partial}\to F_0^*\xrightarrow{\alpha^*}F_1^*, 
\end{equation}
where $r = \rank [\Sym_R (I)]_{d-n-i}$, i.e., $\operatorname{Im} (\partial)(-n) \subset \ker \alpha^*(-n) \cong \mathcal{A}_{i}$. 
\end{thmA}

We note that the above is a complex for any degree of $\mathcal{A}_i$ (equivalently, $[\Sym_R (I)]_{d-n-i}$).
It is natural to ask under what conditions the complex in \cref{eqIntro} is exact. 
We show that for $r \le 2$, this complex is exact provided that $[\Sym_R(I)]_{d-n-i}$ satisfies Serre's condition $(S_{r})$ (\Cref{thmRk1,thmRk2}), and we ask that whether this holds true in general (\Cref{rge3Question}). 
Our theorems provide the differential map $\partial$ and the shift $s_1$ in \cref{eqIntro} explicitly. 
To do this we use a structure theorem of Buchsbaum-Eisenbud \cite{BE74}, and their main theorem and lemmas are our main technical tools in this paper. \\

Our next theorem concerns \Cref{intQuestion}(b). It is a well-known fact that the closed image $\overline{\phi(X)}$ is defined by a single equation, for instance, see \cite[Proposition 2.4]{UV}. Hence, it suffices to study 
\begin{equation}
\mathcal{A}_0 \cong \Hom_S (\Sym(I)_{d-n},S(-n)) = \ker \alpha^*(-n),
\end{equation} 
where $\alpha^*$ as in \cref{eqIntro}. We provide equivalent conditions that the complex in Theorem A is exact. The equivalence $(1)$ and $(4)$ in Theorem B below was first established in {\cite[Corollary 3.7]{KPU16}}.

\begin{thmB} Assume $\kk$ is a field of characteristic zero and $n \ge 3$. Then with $s_1$ as in \cref{eqIntro}, the generating degree of $\mathcal{A}_0$ is at most $s_1 + n$.  
 Furthermore, the following statements are equivalent:
	\begin{enumerate}[\indent$(1)$]
		\item The rational map $\phi$ is birational to its image.	
		\item $\mathcal{A}_0$ is generated in degree $s_1 + n$.
		\item The greatest common divisor of the entries of $\ker \alpha^*$ is 1.
		\item $e(\mathcal{F}(I))= e ({R}/{(g_0,\dots,g_{n-2}):I})$, where $g_0, \dots, g_{n-2}$ are general $\kk$-linear combinations of the $f_i'$s 
		(see \Cref{defGenEle} for the definition of the term general). Here, $e(-)$ denotes the Hilbert-Samuel multiplicity. 
	\end{enumerate}
\end{thmB}
The explicit isomorphism of Jouanolou provides an explicit form of the defining equation from $\ker \alpha^*$, and there are a few other ways to treat this case. 
In \cite{BuseChardinJouanolou}, Bus\'e, Chardin, and Jouanolou achieved this by  analyzing $[\Sym_R(I)]_q$ for $q \gg 0$ and using the determinant of a free resolution of $[\Sym_R(I)]_q$. 
We note that as $q$ becomes larger, the size of the presentation matrix of $[\Sym_R(I)]_q$ grows in the binomial order of $d = \dim R$.
With our approach, one only needs to analyze the presentation matrix of $[\Sym_R(I)]_{d-n}$.
As far as item (1) in the above theorem is concerned, one may apply a theorem of   Doria, Hassanzadeh, and Simis \cite[Section 2.3]{DHS}. 
Their approach is more general, but it requires understanding an additional graded piece $\mathcal{A}_{1}$ of the defining ideal and its Jacobian dual. Lastly, we compare ours with a result by Boswell and Mukundan \cite{BM16}, where the authors use an iterative Jacobian dual. 
This iteration involves computing large size matrices, and it is computationally not as efficient as our approach. However, their statement provides a closed formula for $\mathcal{A}$ in terms of colon ideals. \\

One of the advantages of our approach is that every condition we impose is general, for instance see \cite[p.\ 316]{BJ03} and \Cref{mainThm}. By the semi-continuity theorem \cite[Thm 14.8b]{Eis}, one can see that the exactness of the complex in \cref{eqIntro} is a general condition.
That is, there exists an open subset $U$, which may be empty, in a parameter space such that the fiber of each point of $U$ satisfies this condition. 
The difficult part is to show that such an open subset $U$ is non-empty. In other words, one needs to exhibit an example whose corresponding point belongs to $U$. This turned out to be the most challenging part of our paper. 
We believe that for $d_1, \dots, d_n \in \mathbb{N}$ and a presentation matrix
\begin{equation}
\varphi = \begin{bmatrix}
x_1^{d_1} &  x_1^{d_2} & \cdots & x_1^{d_n} \\
x_2^{d_1} & x_2^{d_2} & \cdots & x_2^{d_n} \\
\vdots & \vdots & \ddots &  \vdots \\
x_n^{d_1} & x_n^{d_n} & \cdots & x_n^{d_n} \\
g_1 & g_2 & \cdots & g_n 
\end{bmatrix},
\end{equation}
where $g_i$'s are symmetric polynomials of degree $d_i$, the ideal generated by the $n \times n$ minors of $\varphi$
would provide a family of examples. As an evidence, we show that a variant of this with $d_1 = 1, d_2 = 2, d_3 \ge 3$ provides a necessary example (\Cref{bigExample}). This allows us to state our last main result in this paper. 

\begin{thmC}[{\Cref{mainThm}}]\label{thmC}
Let $R = \kk[x,y,z]$, where $\kk$ is a field of characteristic $0$. 
Let $M$ be a $4 \times 3$ matrix whose entries are in $R$ and let $I = I_3(M)$. 
If $M$ is general of type $(1,2, \nn)$, where $\nn > 2$ (see Case 2 in \Cref{defGenMat}),
then the defining ideal of the Rees algebra $R[It]$ is minimally generated by \begin{equation}
\def\arraystretch{1.2}
\begin{array}{c;{2pt/2pt}c|c}
\text{~} & \text{bidegree} & \text{number of elements}  \\
\hline
l_1 & (1,1) & 1 \\ \hline
l_2 & (2,1) & 1 \\ \hline
l_3 & (\nn,1) & 1 \\ \hline
\mathcal{A}_0 & (0, 3\nn+2 ) &  1 \\ \hline
\mathcal{A}_{\nn - i} &  (\nn - i, 3i+1 )&   {2 + i \choose 2}  \\ \hline
\mathcal{A}_\nn & (\nn, 3) & 1  

\end{array}
\end{equation}
 for $1 \le i < \nn-1$.  
In particular, the defining ideal is minimally generated by ${\nn+2 \choose 3} + 4$ elements. Here, $l_1,l_2,l_3$ denote the equations defining $\Sym_R(I)$ in $B$, see \Cref{setupSec}. 
\end{thmC}

The paper is organized as follows. In \Cref{secPrel}, we set up the notation and provide preliminaries. 
In \Cref{setupSec}, we explain our setup mentioned in the introduction, the work of Buchsbaum and Eisenbud, and the duality theorem in detail.
In \Cref{sectiononfunctionsinA1}, we prove Theorem A and related statements.
\Cref{sectionOnDegreesofImplicit} is devoted to the equivalence in Theorem B. 
In \Cref{mainResultsec}, we present Theorem C and the promised example.\\

\noindent{\bf Acknowledgment:}
We would like to thank professors Bernd Ulrich and Craig Huneke for their helpful comments and suggestions. 
Also, we owe a lot to the anonymous referee for his/her invaluable comments. The earlier version of the paper had an erroneous definition of the notion of general elements which was kindly pointed out to us by the referee. 

\section{Preliminaries}\label{secPrel}
In this section, we will setup the notation and review some background materials. We refer the reader to \cite{Eis,Har} for basic definitions and notations for algebraic geometry and commutative algebra. 

\subsection{Morphisms between projective spaces and Rees algebras}
Let $\kk$ be a field, $\kk[x_1,\dots, x_n]$, the coordinate ring for $\mathbb{P}^{n-1}_\kk$, and $f_0,\dots,f_n$ homogeneous polynomials of degree $d$ in $\kk[x_1,\dots, x_n]$. Then we have a rational map between projective spaces
\begin{equation}\label{rationalMap}
\phi: \mathbb{P}^{n-1}_\kk {\xdashrightarrow{[f_0:\cdots: f_n]}}  \mathbb{P}^{n}_\kk
\end{equation}
defined by the polynomials $f_0,\dots, f_n$. This map is defined on $\mathbb{P}^{n-1}_\kk \setminus V(I)$ (equivalently, the base locus is $V(I)$).
The image $\ima \phi$ is not a closed subscheme of $ \mathbb{P}^{n}_\kk$ in general. In this article, we study the closed subscheme $\overline{\ima \phi}$ in $\mathbb{P}^{n}_\kk$. 
Let  $\kk[y_0,\dots, y_n]$ be the coordinate ring of $\mathbb{P}^{n}_\kk$. 
Then $\overline{\ima \phi} = \Proj \kk[f_0,\dots, f_n]$, and the (rational) maps between projective schemes 
\begin{equation}
\mathbb{P}^{n-1}_\kk \stackrel{\phi}{\dashrightarrow} \overline{\ima \phi} \subset \mathbb{P}^{n}_\kk
\end{equation}
 corresponds to the maps between $k$-algebras
\begin{equation}
\kk[y_0,\dots,y_n] \twoheadrightarrow \kk[f_0,\dots,f_n] \subset \kk[x_1,\dots,x_n],
\end{equation}
where the first map is defined by $y_i \mapsto f_i$ for $i = 0, \dots, n$. Observe that $\overline{\ima \phi} = \Proj \kk[f_0,\dots, f_n] = V(J)$ for some homogeneous ideal $J$ of $\kk[y_0,\dots, y_n]$\footnote{Such construction holds for arbitrary set of homogeneous ideals of the same degree in $S$, but this is the set up we will work in this paper.}. \\

Let $R = \kk[x_1,\dots, x_n]$, $\m = (x_1,\dots, x_n)R$, and $I = (f_0,\dots, f_n)$. The \textit{Rees algebra} of $I$ is the graded ring
\begin{equation}
R[It] = R \oplus I t \oplus I^2 t^2 \oplus \cdots \subset R[t],
\end{equation}
and \textit{the special fiber ring} of $I$ is the graded ring
\begin{equation}
\mathcal{F}(I) = R[It] \otimes_R R/ \m = R[It]/\m R[It].
\end{equation}

The Rees algebra $R[It]$ is the coordinate ring of the closure of the graph of the rational map in \cref{rationalMap}. Here, the \textit{graph} of a  rational map $\phi: X \dashrightarrow Y$ with base locus $W$ between projective schemes is $\Gamma_{\phi} := \{ (x,y) \in X \times Y \mid y = \phi(x), x \not\in W \}$. 
In addition, if $\deg f_i = d$ for $i = 0,\dots, n$, then the special fiber ring $\mathcal{F}(I)$ is an integral domain and is isomorphic to the subring $\kk[f_0,\dots, f_n]$ of $\kk[x_1,\dots, x_n]$.\\

Let $\Quot(A)$ denote the total ring of fractions of a ring $A$. In the case where $A$ is an integral domain, $\Quot(A)$ is the field of fractions. For fields $F \subset K$, let $[K : F]$ denote the field extension degree. 
 
\begin{remark}[{cf. \cite[Prop. 2.11]{DHS}}]\label{birationalprelim}
The rational map in \cref{rationalMap} is birational to the image if and only if $\kk[f_0,\dots, f_n]$ and $\kk[x_1,\dots, x_n]$ have the same field of fractions. 
\end{remark}

\subsection{Free resolutions and minors of matrices}
Let $R$ be a Noetherian ring. 
In this subsection, we review two theorems of Buchsbaum-Eisenbud on finite free complexes.  One provides a characterization of the acyclicity of a finite free (graded) $R$-complex, and the other one provides a structure theorem for a finite free (graded) acyclic $R$-resolution. \\

Let $\varphi: F \to G$ be a map between finite free $R$-modules of rank $f$ and $g$, respectively. Once we fix ordered bases for $F$ and $G$, we obtain a matrix representation $M$ of $\varphi$, which is an $g \times f$ matrix with entries in $R$. 
For an $m \times n$ matrix $N$ with entries in $R$, let $I_t(N)$ be the ideal generated by $t \times t$ minors of $N$ if $1 \le t \le \min \{ m, n\}$, and we set $I_0(N) = R$ and $I_t (N) = 0$ if $t > \min \{ m,n\}$. 
By abuse of notation, let $I_t(\varphi)$ denote $I_t(M)$, where $M$ is a matrix representation for $\varphi$. 
Matrix representations depend on the choice of bases. However, the ideal $I_t(M)$ does not depend on the choice of bases.
The \textit{rank} of $\varphi$, denoted $\rk \varphi$, is the number $t$ where $I_{t+1}(\varphi) = 0$, but $I_t(\varphi) \neq 0$, and we set $I (\varphi) = I_{\rk \varphi} (\varphi)$. 
For a proper ideal $I$ of $R$, the grade of $I$, denoted $\grade I$ (equivalently, the depth of $I$ in $R$ denoted by $\depth_I R$), is the length of a maximal $R$ regular sequence contained in $I$. It is well-known that the maximal length is independent of regular sequences, cf. \cite[Def. 1.2.11]{BH}.

\begin{thm}[{\cite[Theorem]{BE73}}]\label{thmBE73}
Let $\mathcal{C}_\bullet$ be a finite complex of free $R$-modules of finite rank
\begin{equation}
0 \to F_n \stackrel{\varphi_n}{\to}  \cdots \stackrel{\varphi_2}{\to} F_1 \stackrel{\varphi_1}{\to} F_0.
\end{equation}
Then $\mathcal{C}_\bullet$ is acyclic if and only if for $k = 1,\dots, n$,
\begin{enumerate}[\indent$(1)$]
\item $\rk F_k = \rk \varphi_{k+1} + \rk \varphi_{k}$ and

\item $\grade I(\varphi_k) \ge k$ or $I(\varphi_k) = R$. 
\end{enumerate}
\end{thm}

The map $\varphi : F \to G$ induces maps between exterior powers $\wedge^k \varphi: \wedge^k F \to \wedge^k G$ for any $k$. 
We also denote the image of $\wedge^k \varphi$ by $I_k(\varphi)$. 
(Once we fix bases for $F,G$ and a matrix representation $M$ for $\varphi$, the matrix representation of $\wedge^k \varphi$ is the $k$-minors of $M$ (up to sign). 
Therefore, the image of $\wedge^k \varphi$ in $R$ is $I_k(\varphi)$.)  
For $F$ a free module of rank $f$, an isomorphism $\eta: \wedge^f F \to R$ is called an \textit{orientation} of $F$. We say a finite free module is \textit{oriented} if it is equipped with an orientation. For an oriented finite free module $F$, we have 
$ \wedge^k F \otimes \wedge^{f-k} F \to \wedge^f F \stackrel{\eta}{\to} R$. Hence, we identify $(\wedge^k F)^*$ with $\wedge^{f-k} F$ for oriented free modules. Here, for any $R$-module $L$, $L^* := \Hom_R(L,R)$ denotes the $R$-dual of $L$.

\begin{thm}[{\cite[Theorem 3.1]{BE74}}]\label{BEstatementprelim}
Consider a finite free acyclic $R$-complex\label{BEprelim}
\begin{equation}
0 \to F_n \stackrel{\varphi_n}{\to}  \cdots \stackrel{\varphi_2}{\to} F_1 \stackrel{\varphi_1}{\to} F_0.
\end{equation}
Write $r_i = \rk \varphi_i$. 
For $k = 1, \dots, n$, there exists unique $R$-homomorphism $a_k : R \to \wedge^{r_k} F_{k-1}$ such that 
 \begin{enumerate}[\indent $(1)$] 
   \item $a_n := \wedge^{r_n}\varphi_n: R = \wedge^{r_n} F_n \to \wedge^{r_n} F_{n-1}$, and 
   \item for $k < n$, the diagram
   \begin{equation}
   \xymatrix{
   \wedge^{r_k} F_k   \ar[rr]^{\wedge^{r_k} \varphi_k} \ar[dr]_{a_{k+1}^*}&   &  \wedge^{r_{k}} F_{k-1} \\
   & R \ar[ur]_{a_k}
   }
   \end{equation}
   commutes.
   \item For all $k>1$, $\sqrt{I(a_k)}=\sqrt{I(\varphi_k)}$.
 \end{enumerate}
\end{thm}

\subsection{General property}\label{secGeneral}
In this subsection, we recall the notion of a general property. 
We will follow the section ``general object'' in \cite{Har92}. 
Let $X$ be a variety (or a scheme) parametrized by (closed) points in an (irreducible and reduced) variety $Y$ and $\mathcal P$ a property on $X$. 
We say that $\mathcal P$ is \textit{general} or a \textit{general property} with respect the pair $X$ and $Y$, if the set $\{ p \in Y \mid \text{object parametrized by} p~\text{satisfies}~\mathcal{P} \} \subset Y$ is a dense open subset. 
If $\mathcal{P}$ is a general property and $y \in Y$ satisfies $\mathcal{P}$, then $y$ or the object of $X$ parameterized by $y$ is called \textit{general} or a \textit{general member} with respect to $\mathcal{P}$. 
In the sequel, whenever we use the phrase an object $G$ is general or a general member, it is understood that it refers to a general property for a pair $X$ and $Y$, and  $G$ a general member with respect to this general property. \\

We will use this notion of a general property in the following setup. 
Let $\kk$ be a field of characteristic zero, $R = \kk[x_1, \dots, x_n]$ a polynomial ring, and $A$ another polynomial ring over $\kk$. 
Further, let $B = A \otimes_{\kk} R \cong A[x_1,\dots,x_n]$ and $J$ a homogeneous ideal of $B$ which does not contain any element of $A$ other than zero, i.e., $J \cap A = 0$. 
Here, we set $X = \Proj_Y B/J$, where $Y = \Spec A$. 
Our parameter space $Y$ will be always affine space over a field of characteristic zero, and we will consider only closed points of $Y$.
By abuse of terminology, we say an $R$-ideal $I$ is \textit{general} or \textit{a general member} with respect to some general property $\mathcal{P}$ if $\Proj_\kk R/I$ is a general member with respect to $\mathcal{P}$ for the pair $X$ and $Y$.
We also say a sequence of elements $g_1, \dots, g_s$ of $R$ (or a matrix $M$ whose entries are in $R$) is \textit{general} or \textit{a general member} with respect to some general property $\mathcal{P}$ if the ideal $(g_1,\dots, g_s)$ (or $I_s(M)$ for some fixed $s$) is a general member with respect to $\mathcal{P}$. 
For the sake of completeness, we list the setups of two cases in detail.
We will use Case 1 in \Cref{sectionOnDegreesofImplicit} and Case 2 for \Cref{mainThm} and in its proof.\\

\textbf{Case 1:}\label{defGenEle}  
Let $R =\kk[x_1,\dots, x_n]$ be a polynomial ring over a field of characteristic zero $\kk$ and $f_1,\dots, f_l$ homogeneous polynomials of degree $d \ge 1$ in $R$.
Fix a positive integer $s$. 
Consider the parameter space $A := \kk[ u_{ij} \mid 1 \le i \le l, 1 \le j \le s]$, and for $1 \le j \le s$, set 
\begin{equation}
F_j := u_{1j} f_1 + u_{2j} f_2 + \cdots + u_{lj} f_l,
\end{equation}
in $B := A \otimes_{\kk} R \cong A[x_1,\dots,x_n]$. Write $J = (F_1,\dots, F_s) \subset B$. 
Here $Y = \Spec A$ and $X = \Proj_Y B/J$. 
A general member satisfying some property with respect to the pair $X$ and $Y$ is also often called as \textit{$s$-general elements} or \textit{$s$-general $\kk$-linear combination} of $f_1,\dots, f_l$. \\

\textbf{Case 2:}\label{defGenMat} 
$R = \kk [x_1, \dots, x_n]$ be a polynomial ring over a field of characteristic zero $\kk$. 
Fix positive integers $m,l$ and $d_1, \dots, d_l$. 
For $1 \le j \le l$, let $h_j$ denote the number of monomials of $R$ of degree $d_j$, and define $h := h_1 + \dots + h_l$. 
Our parameter space is $A = \otimes_{\kk} \, A_{ij}$, where $A_{ij} = \kk[ u_{ij,k} \mid  1 \le k \le h_1 ]$ for $1 \le i \le m, 1 \le j \le l$.
In $B := A \otimes_{\kk} R = A[x_1,\dots,x_n]$, consider the $m$ by $l$ matrix ${M}^\sim$ whose $(i,j)$-entry is 
\begin{equation}
u_{ij,1} m_{d_j,1} + u_{ij,2} m_{d_j,2} + \cdots + u_{ij,h_j} m_{d_j,h_j},
\end{equation}
where $\{m_{d_j,1}, \dots, m_{d_j,h_j} \}$ is a fixed monomial basis of $R$ of degree $d_j$. 
Let $J = I_s (M^\sim)$, where $s$ is an integer. 
Here, $X = \Proj_Y B/J$ with $Y = \Spec A$.
We say that an $m \times l$ matrix $M$ with entries in $R$ is \textit{general of type $(d_1, \dots, d_l)$}  if for some $s$, $I_s(M)$ is a general member for some general property with respect to this pair $X$ and $Y$.\\

We apply the general property in Case 2 
to the height of $I_s(M)$ for some integer $s$. 
The open subset defining such a general property corresponds to the complement of the closed subset defined by the ideal $I_e$ in \Cref{corGeneralHeight}(b) (in our notation, $S = B/J$), provided that such complement is non-empty.
We will use this correspondence in the proof of the main theorem (\Cref{mainThm}).  

\begin{theorem}\label{corGeneralHeight}  
Suppose that $A$ is a Noetherian ring and $S = S_0 \oplus S_1 \oplus \cdots$ is a positively graded ring which is finitely generated over $A = S_0$.  
Then we have the following statements.
\begin{enumerate}[$(a)$]
\item {\cite[Theorem 14.8(b)]{Eis}} For any integer $e$, there exists an $A$-ideal $I_e$ (depending on $e$) such that for any maximal ideal $\m$ of $A$, 
\begin{equation}\label{labUSC}
\dim A/\m \otimes_A S \ge e~\quad \iffw~\quad \m \supset I_e.
\end{equation}

\item Write $S = A[x_1,\dots, x_n]/J$. 
For any integer $e$, there exists an $A$-ideal $I_e$ (depending on $e$) such that for any maximal ideal $\m$ of $A$, 
\begin{equation}\label{labUSC}
\hgt J (A/\m) [x_1,\dots, x_n] \le n-e~\quad \iffw~\quad  \m \supset I_e. 
\end{equation}
\end{enumerate}
\end{theorem}
\begin{proof}
We show part (b). Fix a maximal ideal $\m$ of $A$ and write $T =  A/\m \otimes_A A[x_1,\dots,x_n] = (A/ \m)[x_1,\dots, x_n]$. Notice that $T/JT = A/ \m \otimes_A S$. 
Now, part (b) follows from the first part and the following identity (\cite[Theorem 1.8A]{Har})
\begin{equation}
\hgt J T + \dim T/ J T   = \dim T (= n).
\end{equation} 
\end{proof}

We note that {\cite[Theorem 14.8(b)]{Eis}} is stated for prime ideals of $A$. In this paper, we only consider closed points of a parameter space, so we stated part (a) for maximal ideals. 

\begin{remark}
\begin{enumerate}[\indent$(a)$] 
\item We list examples of general properties.
 \begin{itemize} 
  \item For an ideal $I = (f_1,\dots,f_l)$ of grade $\ge s$, the property of the grade of the ideal generated by $s$-general linear combinations of $f_1,\dots,f_l$ being at least $s$ is a general property. 
  
  \item For an ideal $I = (f_1,\dots,f_l)$ minimally generated by the $f_i$'s and $s \le l$, the condition that $s$-general linear combinations of $f_1,\dots,f_l$ are part of minimal generating set is a general property. 
 
\end{itemize}
 
\item An advantage of having a general property is that one may ask finitely many general properties simultaneously since a finite intersection of non-empty dense open subsets remains non-empty dense open.

\item For those who are familiar with algebraic geometry, generic freeness and generic smoothness of $\mathbb{C}$-varieties are examples of general properties \cite[Theorem 14.4]{Eis}.
\end{enumerate}
\end{remark}


\section{Defining ideals of Rees algebras and special fiber rings}\label{setupSec}
Let $R = \kk[x_1,\dots, x_n]$ be a polynomial ring in $n$ variables over a field $\kk$, 
$\m = ( x_1,\dots, x_n)R$ the homogeneous maximal ideal, 
$f_0,\dots, f_n$ homogeneous polynomials of degree $d$ in $R$, 
and $I = (f_0,\dots, f_n)$. 
Consider a homogeneous surjective $R$-linear map $\pi$ 
\begin{equation}\label{piMapDefEqReesAlg}
\pi: R[T_0,\dots,T_n] \to R[It],
\end{equation}
where $T_i \mapsto f_i t$ for $0 \le i \le n$. Since $\mathcal{F}(I) = R[It] / \m R[It]$, the map $\pi$ induces a surjective $\kk$-linear map $\overline{\pi}$ for $\mathcal{F}(I)$
\begin{equation}
\overline{\pi}: \kk[T_0,\dots,T_n] \to \mathcal{F}(I).
\end{equation}

The kernel of $\pi$ is called the \textit{defining ideal} of the Rees algebra of $I$, and 
the kernel of $\overline{\pi}$ is called the \textit{defining ideal} of the special fiber ring $\mathcal{F}(I)$. Recall that $f_0, \dots, f_n$ define a rational map 
\begin{equation}\label{parametrizationmap}
\phi: \Proj (R) = \PP^{n-1} {\xdashrightarrow{[f_0:\cdots: f_n]}}  \mathbb{P}^{n}_\kk,
\end{equation}
$V(\ker \pi)$ defines the closure of the graph of $\phi$ in $\PP^{n-1} \times \PP^{n}$, and $V(\ker \overline{\pi})$ defines the closure of the image of $\phi$ in $ \PP^n$. \\

The study of $\ker \pi$ can be simplified via the symmetric algebra of $I$. For an ideal $I$, the \textit{symmetric algebra} of $I$ is the graded ring
\begin{equation}
\Sym_R(I) := R \oplus I \oplus \Sym_R^2(I) \oplus \cdots,
\end{equation}
where $\Sym_R^k(I)$ denotes the $k$th symmetric power of $I$. By the universal property of $\Sym_R(I)$,  the homogeneous presentation $\pi$ factors through $\Sym_R(I)$
\begin{equation}\label{presSym}
\xymatrix{
R[T_0,\dots, T_n]  \ar[rr]^{\pi} \ar[rd]_{\pi'} && R[It]\\
& \Sym_R(I) \ar[ru]_{\pi''}.
}
\end{equation}
The kernel of $\pi'$ is easy to describe from the graded presentation matrix of $I$: Consider a homogeneous presentation $\varphi$ of the ideal $I$
\begin{equation}\label{presI}
\oplus^m R(-d-d_i) \stackrel{\varphi}{\to} R(-d)^{n+1} \to I \to 0,
\end{equation}
where $d_i$ are positive integers. 
Notice that $\varphi$ is a $n+1 \times m$ matrix. Then $\ker \pi' = I_1 ( [T_0 \, \dots \, T_n] \, \varphi)$. Hence by abuse of notation, we call $\mathcal{A} = (\ker \pi) \Sym_R(I)$ the defining ideal of the Rees algebra. \\

Since $R$ is a graded ring, $R[T_0,\dots, T_n]$ has a natural bi-graded structure; we set $\deg x_i = (1,0)$ and $\deg T_i = (0,1)$. 
Hence $\ker \pi$ is a bi-graded ideal of $R[T_0,\dots, T_n]$, and $\ker \overline{\pi}$ is a homogeneous ideal of $\kk[T_0,\dots, T_n]$. Write $B = R[T_0,\dots, T_n]$ and $S = \kk[T_0,\dots, T_n]$. It is often convenient to view $B$ as $B \cong R \otimes_\kk S = \kk[x_1, \dots, x_n ] \otimes_\kk \kk[T_0,\dots, T_n]$. Hence $B$ is free as an $S$-module. For a graded module $M$, $M(a)$ denotes the grade shift by $a$. That is $[M(a)]_{i} = M_{a+i}$ for any $i \in \mathbb{Z}$. Similarly, for a bigraded module $M$, $M(a,b)$ means $[M(a,b)]_{(i,j)} = M_{(a+i, b+j)}$ for any $i,j \in \mathbb{Z}$.\\

Our main theorems are stated in the same hypothesis of the following proposition. The statement (as well as its proof) is well-known. We present this proposition to fix the notation, and we will use it as a quick reference for our setup. 

\begin{prop}\label{setupProp}
Let $R = \kk[x_1,\dots, x_n]$ be a polynomial ring in $n$ variables over a field $\kk$, and $I$ an $R$-ideal, and $\m = (x_1,\dots, x_n)R$. Assume that $I$ is codimension $2$ perfect, that the degrees of the entries of the columns of a presentation matrix $\varphi$ of $I$ in \cref{presI} are $d_1 \le d_2 \le  \cdots \le d_n$, and that $\mu(I_p) \le \dim 
R_p$ for all $p \in \Spec(R) \setminus \{\m \}$. Then we have the following:
\begin{enumerate}[\indent~$(1)$]
\item The complex 
$0 \to \oplus^n R(-d_i) \stackrel{\varphi}{\to} R^{n+1} \to I (d)\to 0$
is exact.  
\item $I$ is generated in degree $d = \Sigma^n d_i$.
\item $\Sym_I (R)$ is a complete intersection.
\item 
$ \mathcal{A} = H^0_\m (\Sym_I (R)),$
 where $H^0_\m (\Sym_I (R))$ denotes the 0-th local cohomology module of $\Sym_I(R)$ with support in $\m$. 
\end{enumerate}
\end{prop}

\begin{proof}
$(1)$ and $(2)$ follow from the Hilbert-Burch theorem \cite[Theorem 20.15]{Eis} since $I$ is a codimension two perfect graded ideal, $(3)$ follows from the fact that $\dim \Sym_R(I) = n + 1$ and $\ker \pi'$ in \cref{presSym} is generated by $n$ elements, and $(4)$ follows from \cite[Theorem 2.6]{HSV1}. 
\end{proof}

\begin{remark}\label{Gs}
For a ring and an ideal $I$, we say that $I$ satisfies the condition $(G_s)$ if  
$\mu(I_p) \le \dim 
R_p$ for all $p \in \Spec(R)$ such that $\dim R_p \lneq s$. Hence the ideal $I$ in \Cref{setupProp} satisfies $(G_{\dim R})$ condition. There are plenty of ideals satisfying this condition, e.g., complete intersection ideals, (homogeneous) ideals primary to the homogeneous maximal ideal, and this condition can be checked with the Fitting ideals of $I$ (that is $I$ satisfies $(G_s) \iff$ $\hgt\operatorname{F}_i(I) \ge i + 1$ for $1 \le i \le s-1$).
\end{remark}

\begin{remark}\label{pieceRes}
With the notation and hypothesis of \cref{setupProp}, let $l_1, \dots, l_n$ be in $B = R[T_0,\dots, T_n]$ such that $[ \, l_1 \, \dots \, l_n] = [\, T_0 \, \dots \, T_n] \, \varphi$; hence $\Sym_R(I) = B / (l_1, \dots, l_n)$, and $l_1, \dots, l_n$ form a bi-homogeneous regular sequence in $B$. Then the Koszul complex $\mathcal{K}:= \mathcal{K} ( \underline{l_i} ; B)$ is a bi-graded $B$-resolution for $\Sym_R(I)$; 
\begin{equation}
\mathcal{K}: 0 \to B( -\Sigma^n d_i, -n) \to  \cdots \to \oplus^n B(-d_i, -1) \to B. 
\end{equation}

Since $B$ is a free graded $S = k[T_0,\dots, T_n]$-module, $\mathcal{K}$ is a graded free $S$-resolution of $\Sym_R(I)$. 
Each component of $\mathcal{K}$ is not of finite rank as an $S$-module. However, for each $k \in \mathbb{Z}$, $\mathcal{K}_k := \mathcal{K}_{(k,*)}$ is a finite free graded $S$-resolution for $[\Sym_R(I)]_{(k,*)}$, and each component of $\mathcal{K}_k$ is of finite rank as $S$-module;
\begin{equation}\label{koszulSk}
\mathcal{K}_k : 0 \to F_n \stackrel{}{\to} \cdots \to F_1 \stackrel{}{\to} F_0,
\end{equation} 
where $F_i=\bigoplus_{1\leq j_1\leq j_2\leq\cdots\leq j_i\leq n}S(-i)^{k-(d_{j_1}+\cdots+d_{j_i})+n-1\choose n-1}$. Furthermore, these $S$-resolutions are linear resolutions, i.e., the non-zero entries of the differential maps are of degree 1. \\
\end{remark}

Recall the notation that for a bigraded module $M = \oplus_{i,j \in \mathbb{Z}} M_{i,j}$ in $B$, $M_i := \oplus_{j \in \mathbb{Z}} M_{i,j}$.
Hence $\Sym(I)_k=[\Sym_R(I)]_{(k,*)}$ and $\mathcal{A}_k=[\mathcal{A}]_{(k,*)}$. In the sequel, when we use a single grading for $B = R \otimes S$, we will always follow this convention.

\begin{thm}[{Grothendieck, \cite[Section 3.6]{Jou97}, \cite[Theorem 2.4]{KPU17} }]\label{perfPair}
For each $i \in \{ 0,\dots, \delta \}$, there exists an isomorphism of finitely generated graded $S$-modules
\begin{equation}
\mathcal{A}_i\cong \Hom_S (\Sym(I)_{\delta-i},S(-n)),
\end{equation}
where $\delta= d-n = d_1 + \dots + d_n - n$. 
\end{thm}
In \cite[Section 3.6]{Jou97}, Jouanolou describes the isomorphism in terms of Morley forms (for instance, see \cite[Chapter 4]{KPU17}).
Therefore, an explicit computation of $\Hom_S (\Sym(I)_{\delta-i},S(-n))$ leads to a generating set, not only their bidegrees, of $\mathcal{A}_i$. 
In the following two sections, we study $\mathcal{A}_i$ and give a generalized method to compute a generating set for $\mathcal{A}_i$.

\section{The dual generators of the defining ideal}\label{sectiononfunctionsinA1} 
We will adapt the notation in \Cref{setupSec} and the notation and hypothesis of \Cref{setupProp}. 
The graded free $S$-resolution of  $\Sym(I)_{\delta-i}$ in \Cref{pieceRes} is
\begin{equation}\label{delta-ires} 
\mathcal{K}_{\delta-i} : 0 \to F_m \stackrel{\alpha_m}{\to} \cdots \to F_1 \stackrel{\alpha_1}{\to} F_0,
\end{equation} 
where $F_t=\bigoplus_{1\leq j_1\leq j_2\leq\cdots\leq j_i\leq n}S(-t)^{\delta-i-(d_{j_1}+\cdots+d_{j_i})+n-1\choose n-1}$ for $t = 0,\dots, m$.
Let $r_t=\rk\alpha_t$ and $f_t=\rk F_t$. By \Cref{thmBE73}(1), $f_t = r_t + r_{t+1}$ for $t = 1,\dots, m$, 
and by \Cref{perfPair}, $\mathcal{A}_i \cong \Hom_S(\Sym(I)_{\delta-i},S(-n))=  \ker\alpha_i^*(-n)$. 
For the following construction, it is worth mentioning that the above minimal graded resolution is linear. \\

\begin{remark}\label{reviewBE74}
In order to apply the theorems in \cite{BE74} to our set up, we need to specify shifts. It is not a hard task to do, but for the convenience of the reader and to set up the notation, we review their construction below. We also note that their constructions are for projective modules, but in our paper we only need their theorems for (graded) free modules. Hence our review is written for \textit{graded} free modules of their work.\\

\begin{enumerate}[\indent$(1)$]
\item  
 First for $0\leq t\leq m$, fix a basis for $F_t$. This enables us to have a matrix representation of $\alpha_t$ and $\wedge^{r_t}\alpha_t$ for $1\leq t\leq m$, respectively.
 Henceforth whenever we talk about these maps, we use their matrix representations. We  use \Cref{BEprelim} on \cref{delta-ires} to construct maps $a_t$ for $t=1,\dots,m$. Since $\alpha_m:F_m\rightarrow F_{m-1}$ is an injective map, 
 the entries of $a_m=\wedge^{r_m}\alpha_m=\wedge^{f_m}\alpha_m$ are of degree $r_m$. 
 Thus $a_m:S(-r_m m))\rightarrow \wedge^{r_m}F_{m-1}$. (Recall that since \cref{delta-ires} is a linear resolution, the shift of $F_t$ is $-t$ for all $t = 0,\dots,m$.)
Using \Cref{BEprelim}, we have have the commutative diagram
 \begin{equation}
 \xymatrix{
 	\wedge^{r_{m-1}}F_{m-1}\ar[rr]^{\wedge^{r_{m-1}}\alpha_{m-1}}\ar[rd]_{a_m^*}	& & \wedge^{r_{m-1}}F_{m-2}\\
 	& S(-r_{m-1}(m-1) + r_m)\ar[ru]_{a_{m-1}}&
 }
 \end{equation}
 Notice that $a_m^*$ is a row matrix whereas $a_{m-1}$ is a column matrix; 
 \begin{align}
 a_{m}^*=[a^*_{m,1}\cdots a^*_{m,{f_{m-1}\choose r_{m}}}] \qquad a_{m-1}=\begin{bmatrix}
 a_{m-1,1}\\
 \vdots\\
 a_{m-1, {f_{m-2}\choose r_{m-1}}}
 \end{bmatrix}.
 \end{align}

 Then one has the factorization 
 \begin{equation}
 \wedge^{r_{m-1}}\alpha_{m-1}=a_{m-1}\circ a_m^*.
 \end{equation}
 Not all of the entries of $a^*_{m}$ is zero as $a_m^*=(\wedge^{r_m}\alpha_m)^*$ and $\rk\alpha_m=r_m$. Thus there exists an entry, say $a^*_{m,u}\neq 0$. 
 Now to compute $a_{m-1}$, consider the $u$-th column of $\wedge^{r_{m-1}}\alpha_{m-1}$ and divide its entries by $a^*_{m,u}$. From this we deduce that the entries of $a_{m-1}$ are of degree $r_{m-1}-r_m$ and that $a_{m-1}:S(-r_{m-1}(m-1)+r_m)\rightarrow\wedge^{r_{m-1}}F_{m-2}$. 
Iteratively, one can see that 
 $a_1:S(-\sum_{p=1}^m(-1)^{p-1}r_p )\to \wedge^{r_1}F_0$.

\item Let $s_t=r_t \cdot t - \sum_{p=t+1}^m (-1)^{p-t-1}r_p$. Then $a_t:S(-s_t )\rightarrow \wedge^{r_{t}} F_{t-1}$ for $t=1,\dots,m$. 
In particular, $s_1 = r_1 - r_2 + \cdots +(-1)^{m-1}r_m = \sum_{p=1}^m (-1)^{p-1} r_p$.

\item The map $\alpha_1:F_1\rightarrow F_0$ induces the map $F_0^*\otimes F_1\rightarrow S$. By dualizing it, we obtain the map $\widetilde{\alpha_1}:S\rightarrow F_0\otimes F_1^*$. Consider the map
\begin{equation}\label{mapDelta}
\partial:\wedge^{f_0-1}F_0\rightarrow\wedge^{f_0}F_0\otimes F_1^*,
\end{equation}
which is the composition of the following maps
\begin{equation}
\wedge^{f_0-1}F_0 \cong \wedge^{f_0-1}F_0\otimes S\xrightarrow{\ide\otimes \widetilde{\alpha_1}}\wedge^{f_0-1}F_0\otimes F_0\otimes F_1^*\xrightarrow{m_\wedge \otimes \ide}\wedge^{f_0}F_0\otimes F_1^*,
\end{equation}
where $m_\wedge:\wedge^{f_0-1}F_0\otimes F_0\rightarrow \wedge^{f_0}F_0$ is the usual multiplication in the exterior algebra $\wedge F_0$.
Fix an orientation $\eta$ for $F_0$ and let $a_1:S(-s_1) \rightarrow \wedge^{r_1}F_0$ be the map in \Cref{BEprelim}(2). This map $\partial$ is used in the proof of \Cref{thm1}.
\end{enumerate}
\end{remark}

\begin{thm}\label{thm1}
With the hypothesis of  \Cref{setupProp}, for any $i \in \{ 0, \dots, \delta \}$ and $\alpha_1$ in the graded free $S$-resolution of $\Sym_R (I)_{\delta-i}$ in \cref{delta-ires}, 
we have the following statements.
\begin{enumerate}[\indent$(1)$]
\item 
The following is a complex of graded free $S$-modules
\begin{equation}
\wedge^{f_0-r_1-1}F_0\otimes S(-s_1)\xrightarrow{\eta\circ m_\wedge\circ(\ide\otimes a_1)}F_0^*\xrightarrow{\alpha_1^*}F_1^*,
\end{equation}
where $a_1$ is the map in \Cref{BEprelim}. In particular, if $f_0 = r_1 + 1$, then 
the image of $a_1$ is in $\ker \alpha_1^*$. 

\item In addition, if $\rk\Sym(I)_{\delta-i}>1$, then  
\begin{equation}
\wedge^{f_0-r_1-1}F_1\otimes S(-s_1)\xrightarrow{\wedge^{f_0-r_1-1}\alpha_1\otimes \ide}\wedge^{f_0-r_1-1}F_0\otimes S(-s_1)\xrightarrow{\eta\circ m_\wedge\circ(\ide\otimes a_1)}F_0^*\xrightarrow{\alpha_1^*}F_1^*
\end{equation}
is a complex of graded free $S$-modules.
\end{enumerate}
\end{thm}

\begin{proof}
\noindent (1): \cite[Lemma 3.2(a)]{BE74} (the map $\partial$ in \cref{mapDelta} is equal to $d_{f_0-r_1-1}^{\, \alpha_1}$) implies that the composition 
\begin{equation}\label{lem3.2a}
\wedge^{f_0-r_1-1}F_0\otimes \wedge^{r_1}F_1\xrightarrow{m_\wedge\circ(\ide\otimes\wedge^{r_1}\alpha_1)}\wedge^{f_0-1}F_0\xrightarrow{\partial}\wedge^{f_0}F_0\otimes F_1^*
\end{equation}
is zero.\\

We identify $\wedge^{f_0-1}F_0 \cong F_0^*$ and $\wedge^{f_0}F_0\cong S$ following the fixed orientation $\eta$ of $F_0$. 
Lemma 3.2(c) in \cite{BE74} implies that the following diagram commutes up to sign: 
\begin{equation}\label{lem3.2c}
\xymatrix{
	\wedge^{f_0-1}F_0 \ar[rr]^{\partial} \ar[d]_{\eta} && \wedge^{f_0}F_0\otimes F_1^* \ar[d]^{\eta \otimes \ide}\\
	\wedge^1 F_0^* \ar[rr]^{(m_\wedge\circ(\ide\otimes \alpha_1))^*} &&  \wedge^0 F_0^* \otimes F_1^*.
}
\end{equation}
Notice that in the above diagram, the vertical maps are isomorphisms and the bottom map, i.e., $(m_\wedge\circ (1\otimes \alpha_1))^*$ is a composition of the following maps
\begin{align}
\wedge^1F_0^*\xrightarrow{m_\wedge^*}\wedge^0F_0^*\otimes \wedge^1F_0^* \xrightarrow{(\ide\otimes \alpha_1)^*}\wedge^0F_0^*\otimes \wedge^1F_1^*.
\end{align}
Since $\wedge^0F_0^* \cong S$, the map $(m_\wedge\circ (1\otimes \alpha_1))^*$ can be identified with the map
\begin{equation}\label{lemAlpha}
F_0^* \stackrel{\alpha_1^*}{\longrightarrow} F_1^*.
\end{equation}

From \eqref{lem3.2a}, \eqref{lem3.2c}, and \eqref{lemAlpha}, we conclude that 
\begin{equation}\label{seqTop}
\wedge^{f_0-r_1-1}F_0\otimes \wedge^{r_1}F_1\xrightarrow{\eta\circ m_\wedge\circ(\ide\otimes\wedge^{r_1}\alpha_1)}F_0^*\xrightarrow{\alpha_1^*} F_1^*
\end{equation}
is a complex. Thus $\ima (\eta\circ m_\wedge\circ(\ide\otimes \wedge^{r_1}\alpha_1))\subseteq\ker\alpha_1^*$.\\

With the commutative diagram in Theorem \ref{BEprelim}(a)
\begin{equation}
\xymatrix{
	\wedge^{r_1}F_1\ar[rd]_{a_2^*}\ar[rr]^{\wedge^{r_1}\alpha_1} & &\wedge^{r_1}F_0\\
	& S(-s_1),\ar[ur]_{a_1} &	
}
\end{equation}
the first morphism in \cref{seqTop} factors as follows
\begin{equation}\label{thm1commdiag2}
\xymatrix{
	\wedge^{f_0-r_1-1} F_0 \otimes \wedge^{r_1} F_1\ar[rr]^{\eta\circ m_\wedge\circ (\ide\otimes \wedge^{r_1}\alpha_1)} \ar[d]_{\ide\otimes a_{2}^*} && F_0^*\ar[r]^{\alpha_1^*} & F_1^*
	\\
	\wedge^{f_0-r_1-1}F_0\otimes S(-s_1)  \ar[r]^{\ide \otimes a_1}& \wedge^{f_0-r_1-1}F_0 \otimes \wedge^{r_1} F_0 \ar[r]^{\qquad m_\wedge}  & \wedge^{f_0-1}F_0 \ar[u]^{\eta}.
}
\end{equation}
 Now let $J=\ima a_2^*$ be an $S$-ideal (with a shift). Then $\ima \ide\otimes a_2^*=\wedge^{f_0-r_1-1}F_0\otimes J$. 
By \Cref{BEprelim}(c), $\sqrt{I(\alpha_2)}=\sqrt{I(a_2)}$. So, the ideal $J$ has a positive grade. 
Since the composition of maps in the top row is zero and  the above diagram commutes, we have 
 \begin{align}\label{thm1ProofRepeat}
 0&=(\alpha_1^*\circ \eta\circ m_\wedge\circ(\ide\otimes \wedge^{r_1}\alpha_1))(\wedge^{f_0-r_1-1}F_0\otimes \wedge^{r_1}F_1)\\
 &=(\alpha_1^*\circ \eta\circ m_\wedge\circ(\ide\otimes a_1)\circ(\ide\otimes a_2^*))(\wedge^{f_0-r_1-1}F_0\otimes \wedge^{r_1}F_1)\\
 &=(\alpha_1^*\circ \eta\circ m_\wedge\circ(\ide\otimes a_1))\circ(\ide\otimes a_2^*)(\wedge^{f_0-r_1-1}F_0\otimes \wedge^{r_1}F_1)\\
  &=(\alpha_1^*\circ \eta\circ m_\wedge\circ(\ide\otimes a_1))(\wedge^{f_0-r_1-1}F_0\otimes J(-s_1))\\
  &=J(\alpha_1^*\circ \eta\circ m_\wedge\circ(\ide\otimes a_1))(\wedge^{f_0-r_1-1}F_0\otimes S(-s_1)).
 \end{align}
 The last equality holds as the maps are $S$-module homomorphisms and tensor products are over $S$. Since the last equality holds in the free module $F_1^*$, and $J$ has positive grade, $\alpha_1^*\circ \eta\circ m_\wedge\circ(\ide\otimes a_1)=0$. Thus, 
 \begin{equation}
 \wedge^{f_0-r_1-1}F_0\otimes S(-s_1)\xrightarrow{\eta\circ m_\wedge\circ(\ide\otimes a_1)}F_0^*\xrightarrow{\alpha_1^*}F_1^*
 \end{equation}
 is a complex of graded free $S$-modules. \\
 
\noindent (2):
By item (1), it suffices to show that $[\eta\circ m_\wedge\circ (\ide\otimes a_1)]\circ [\wedge^{f_0-r_1-1}\alpha_1\otimes \ide_S ]=0$. We first extend the complex in the top row of \cref{thm1commdiag2} to 
\begin{equation}\label{cor1complex1}
	\wedge^{f_0-r_1-1}F_1\otimes \wedge^{r_1}F_1\xrightarrow{\wedge^{f_0-r_1-1}\alpha_1\otimes \ide_1}\wedge^{f_0-r_1-1}F_0\otimes \wedge^{r_1}F_1\xrightarrow{\eta\circ m_\wedge\circ(\ide\otimes\wedge^{r_1}\alpha_1)}F_0^*\xrightarrow{\alpha_1^*} F_1^*.
	\end{equation}
Since \cref{delta-ires} is a free $S$-resolution for $\Sym(I)_{\delta-i}$, we have $\rk\Sym(I)_{\delta-i}= \rk F_0 - \rk \alpha_1 = f_0-r_1 > 1$ by the hypothesis. Equivalently, we have $f_0 -1 > r_1$.  
Hence from
\begin{equation}
	\eta\circ m_\wedge\circ(\ide\otimes\wedge^{r_1}\alpha_1)\circ(\wedge^{f_0-r_1-1}\alpha_1\otimes \ide_1)=\eta\circ m_\wedge\circ (\wedge^{f_0-r_1-1}\alpha_1\otimes \wedge^{r_1}\alpha_1)=\eta\circ \wedge^{f_0-1}\alpha_1 = 0,
	\end{equation}
we conclude that \cref{cor1complex1} is a complex.\\

With the factorization $\wedge^{r_1} \alpha_1 = a_1 \circ a_2^*$,  we obtain the following a commutative diagram 
\begin{equation}
 \xymatrix{
\wedge^{f_0-r_1-1}F_1\otimes \wedge^{r_1}F_1\ar[rr]^{\hspace{1cm}\wedge^{f_0-r_1-1}\alpha_1\otimes \ide_1\hspace{1cm}}\ar[d]^{\ide_2\otimes a_2^*}&&\wedge^{f_0-r_1-1}F_0\otimes \wedge^{r_1}F_1\ar[rr]^{\hspace{1cm}\eta\circ m\circ(\ide\otimes\wedge^{r_1}\alpha_1)}& \hspace{1cm}&F_0^*\ar[r]^{\alpha_1^*} &F_1^*.\\
	 	\wedge^{f_0-r_1-1}F_1\otimes S(-s_1)\ar[rr]^{\wedge^{f_0-r_1-1}\alpha_1\otimes \ide_S}&&\wedge^{f_0-r_1-1}F_0\otimes S(-s_1)\ar[rru]_{\eta\circ m_\wedge\circ(\ide\otimes a_1)}& \hspace{1cm}& &
	 }
	 \end{equation}	 
Let $J=\ima a_2^*$ be an $S$-ideal. Then the argument used to prove Item (1) (from \cref{thm1ProofRepeat}) shows that $\eta\circ m_\wedge\circ(\ide\otimes a_1)\circ \wedge^{f_0-r_1-1}\alpha_1\otimes \ide_S=0$. 
This complete the proof of the theorem.
\end{proof}

We present two examples demonstrating \Cref{thm1}(b). It is worth mentioning that in the first example, the corresponding complex is exact whereas in the second example, it is not. 

\begin{example}\label{thm1Example}
\begin{enumerate}[\indent$(a)$]
\item (cf. \Cref{mainThm})
	Let $R=\kk[x,y,z]$ and $I=I_3(\varphi)$, where 
	\begin{equation}
	\varphi=\begin{bmatrix}
	x & y^2 & z^3\\
	y & z^2 & yz^2\\
	z & x^2 & y^3\\
	0 & xy+yz+xz & 0
	\end{bmatrix}.
	\end{equation}
It is easy to verify that this example satisfies the hypothesis of \Cref{setupProp}.
	Consider $B=R[T_0,T_1,T_2,T_3]$ and let $\mathcal{K}$ denote the Koszul complex on $l_1,l_2,l_3$ where $[l_1~l_2~l_3]=[T_0~T_1~T_2~T_3]\cdot\varphi$. Since $\mathcal{K}$ is a graded free resolution of $\Sym(I)$, we can extract the $S$-resolution of $\Sym(I)_1$
	\begin{equation}
	0\rightarrow S(-1)\xrightarrow{\alpha_1}S^3\rightarrow\Sym(I)_1\rightarrow 0,
	\end{equation}
	where $\alpha_1 = \begin{bmatrix}T_0&T_1&T_2\end{bmatrix}^t$.
	Here, we have $f_0 = 3, f_1=1, s_1 = r_1=1$. \Cref{thm1}(b) shows that 
\begin{equation} 
\wedge^1F_1\otimes S(-s_1)\xrightarrow{\wedge^1\alpha_1\otimes \ide} \wedge^1 F_0\otimes S(-s_1)\xrightarrow{m_\wedge\circ(\ide\otimes a_1)}\wedge^1F_0\cong^\eta F_0^*\xrightarrow{\alpha_1^*}F_1^*
\end{equation}
 is a complex. For this example, we will verify that $\alpha_1^*\circ \eta\circ m_\wedge\circ(\ide\otimes a_1)=0$ and $\eta\circ m_\wedge\circ(\ide\otimes a_1)=0\circ \wedge^1\alpha_1\circ\ide=0$.\\
	
	Fixing a basis for $F_0,F_1$ as $\{e_1,e_2,e_3\},\{g_1\}$ respectively, we can fix a basis $\{e_1\wedge e_2,e_1\wedge e_3,e_2\wedge e_3 \}$, $\{g_1\}$ for $\wedge^1F_0$ and $\wedge^1F_1$, respectively. A basis $\{e_1\otimes g_1,e_2\otimes g_1,e_3\otimes g_1 \}$ of $\wedge^1F_0\otimes \wedge^1F_1$ gives us a following matrix representation of $\eta\circ m_\wedge\circ(\ide\otimes a_1)$
	\begin{equation}\label{rk2Example}
	\begin{bmatrix}
		0 & T_2 & -T_1\\
		-T_2 & 0 & T_0\\
		T_1 & -T_0 & 0
	\end{bmatrix}.
	\end{equation}
As $\wedge^1\alpha_1\otimes \ide \cong \alpha_1$ and $\alpha_1^*$ is the transpose $\alpha_1$, we have 
	\begin{align}
	\wedge^1F_1\otimes S(-s_1) 
	\xrightarrow{\begin{bmatrix}T_0\\T_1\\T_2\end{bmatrix}} \wedge^1F_0\otimes S(-s_1)
	\xrightarrow{\begin{bmatrix} 	0 & T_2 & -T_1\\  -T_2 & 0 & T_0\\ T_1 & -T_0 & 0 \end{bmatrix}} F_0^*\xrightarrow{\begin{bmatrix}T_0 & T_1 & T_2\end{bmatrix}} F_1^*.
	\end{align}
Then it is easy to see that this is a complex. 
In fact, the above complex is a graded minimal free $S$-resolution of $\coker\alpha_1^*$
\begin{equation}
0\rightarrow \wedge^1F_1\otimes S(-s_1)\to 
\wedge^1F_0\otimes S(-s_1)\to 
F_0^* \to F_1^*\rightarrow \coker\alpha_1^*\rightarrow 0.
\end{equation}

\item 
	Let $R=\kk[x,y,z]$ and $I=I_3(\varphi)$, where 
	\begin{equation}
	\varphi=\begin{bmatrix}
	x^2 & 0 & x^4\\
	y^2 & x^2 & y^4\\
	z^2 & y^2 & z^4\\
	0 & z^2 & x^3z
	\end{bmatrix}.
	\end{equation}
It is easy to verify that this example satisfies the hypothesis of \Cref{setupProp}. We will use the same notation as in part (a). The number $\delta$ is  $d - n= 8-3 = 5$. The graded free $S$-resolution of $\Sym(I)_\delta$ is
	\begin{equation}
	0\rightarrow S(-2)^3 \xrightarrow{\alpha_2} S(-1)^{23} \xrightarrow{\alpha_1} S^{21 }\rightarrow\Sym(I)_5 \rightarrow 0.
	\end{equation}
	Here, we have $f_0 = 21, f_1=23, f_2 = 3, \rk \Sym(I)_5 = 1$ and $r_1=20, r_2 = 3$. Therefore, $f_0-r_1-1 = 21 - 20 - 1 = 0$, $s_1 = r_1 - r_2 = 17$. \Cref{thm1}(a) implies that 
\begin{equation} 
\wedge^0 F_0\otimes S(-s_1)\xrightarrow{m_\wedge\circ(\ide\otimes a_1)}\wedge^1F_0\cong^\eta F_0^*\xrightarrow{\alpha_1^*}F_1^*
\end{equation}
 is a complex. However, this complex does \textit{not} extend to a free resolution of $\coker\alpha_1^*$ (cf. \Cref{thmRk1,thm5}).
\end{enumerate}
\end{example}

As the examples above demonstrate the complexes in \Cref{thm1} are not exact in general. It is natural to ask which conditions guarantee the exactness of these complexes. A positive answer to this question provides part of minimal generating equations for the defining ideal of the corresponding Rees algebra by Jouanolou.
In the following theorems, we provide sufficient conditions for the exactness in the case where $\Sym(I)_k$ is of rank 1 or 2, respectively.

\begin{theorem}\label{thmRk1}
For a fixed integer $i$, consider the complex 
	\begin{equation}
	\wedge^{f_0-r_1-1}F_0\otimes S(-s_1)\xrightarrow{\eta\circ m_\wedge \circ(\ide\otimes a_1)}F_0^*\xrightarrow{\alpha_1^*}F_1^*,
	\end{equation}
in \Cref{thm1}. 
Assume that $\rk \Sym(I)_{\delta-i}=1$. Then 
the complex is exact if and only if $\grade I(a_1)\geq 2$.
\end{theorem}

\begin{proof}
Since $\rk \Sym(I)_{\delta-i}=1$, we have $f_0-r_1-1=0$. Therefore, the complex is isomorphic to the complex
\begin{equation}\label{rk1complex}
S(-s_1)\xrightarrow{\eta \circ a_1}F_0^*\xrightarrow{\alpha_1^*}F_1^*. 
\end{equation}
Notice that $\rk \alpha_1^* \ge 1$ and $\grade I(\eta(a_1)) = \grade I(a_1)$. Hence by \Cref{thmBE73}, the complex is acyclic if and only if $\grade I( a_1)\geq 2$.
\end{proof}

\begin{theorem}\label{thmRk2}
For a fixed integer $i$, consider the complex 
	\begin{equation}
	\wedge^{f_0-r_1-1}F_0\otimes S(-s_1)\xrightarrow{\eta\circ m_\wedge \circ(\ide\otimes a_1)}F_0^*\xrightarrow{\alpha_1^*}F_1^*,
	\end{equation}
in \Cref{thm1}. 
Assume that $\rk \Sym(I)_{\delta-i}=2$. Consider the following statements:
\begin{enumerate}[\indent$(a)$]
\item $\Sym(I)_{\delta-i}$ satisfies Serre's condition $(S_2)$.
\item $\grade I(\alpha_t) \ge t+2$ for $t = 1,\dots, m$. 
\item The complex above is exact. 
\end{enumerate}

Then we have $(a) \Rightarrow (b) \Rightarrow (c)$. 
\end{theorem}

Before proving the theorem, we present a lemma which explains the relationship between the conditions $(a)$ and $(b)$.

\begin{lemma}\label{lemRk2} With the setup of \Cref{thmRk2}, we have a complex of graded free $S$-modules
\begin{equation}\label{resolution of coker}
 0\rightarrow F_m\otimes S(-s_1)\xrightarrow{\alpha_m\otimes\ide} \cdots \to F_1\otimes S(-s_1)\xrightarrow{\alpha_1\otimes \ide} F_0\otimes S(-s_1)\xrightarrow{\eta\circ m_\wedge\circ(\ide\otimes a_1)}F_0^*\xrightarrow{\alpha_1^*}F_1^*.
\end{equation}
Furthermore, this complex is acyclic if $\Sym(I)_{\delta-i}$ satisfies Serre's condition $(S_2)$.
\end{lemma}

\begin{proof}
Recall that 
\begin{equation}
 0\rightarrow F_m \xrightarrow{\alpha_m} F_{m-1} \rightarrow\cdots\rightarrow F_2 \xrightarrow{\alpha_2} F_1\xrightarrow{\alpha_1} F_0
\end{equation}
is a graded minimal free $S$-resolution of $\Sym(I)_{\delta-i}$, and so is its shift. 
Hence for the first part of the statement, it suffices to show that 
\begin{equation}
F_1\otimes S(-s_1)\xrightarrow{\alpha_1\otimes \ide} F_0\otimes S(-s_1)\xrightarrow{\eta\circ m_\wedge \circ(\ide\otimes a_1)}F_0^*\xrightarrow{\alpha_1^*}F_1^*
\end{equation}
 is a complex. But this follows from \Cref{thm1}(b) with $f_0 - r_1 - 1 = \rk Sym(I)_{\delta-i} - 1 = 1$.\\

We first describe the map $\eta\circ m_\wedge\circ(\ide\otimes a_1): F_0 \otimes S(-s_1) \to F_0^*$. By choosing a basis $\{ e_1,\dots, e_{f_0} \}$ of $F_0$, this can be described as a square matrix. As $a_1 : S(-s_1) \to \wedge^{r_1} F_0 \cong \wedge^{f_0-2} F_0^*$, by abuse of notation, we list the basis for $\wedge^{r_1} F_0$ as $e_i^* \wedge e_j^*$ and $a_1 = \oplus c_{i,j} e_i^* \wedge e_j^*$, where $c_{i,j} \in S$. 
With this notation, one sees that $\eta\circ m_\wedge\circ(\ide\otimes a_1)$ a skew symmetric matrix whose $i,j$th entry for $i<j$ is $c_{i,j}$ up to sign. Hence we have $\sqrt{I_2 (\eta\circ m_\wedge\circ(\ide\otimes a_1))} \supset I_1 (a_1)$ and $\grade I(\eta\circ m_\wedge\circ(\ide\otimes a_1)) \ge \grade I_1 (a_1) \ge \grade I(\alpha_1)$. 
Furthermore, as $\grade \alpha_1^* = \grade \alpha_1$, by \Cref{thmBE73}, the complex in \cref{resolution of coker} is exact if and only if $\grade \alpha_m \ge m + 2$ for $i = 1,\dots, m$. \\
 
It is a well-known theorem of Auslander and Bridger \cite[Theorem 4.25]{AB} that a module $M$ is $k$th syzygy if and only if $M$ satisfies Serre's condition $(S_k)$. 
Hence $\Sym(I)_{\delta-i}$ satisfies $(S_2)$ if and only if for some free $S$-modules $F_{-1}, F_{-2}$, the following complex is acyclic 
\begin{equation}
 0\rightarrow F_m \xrightarrow{\alpha_m} F_{m-1} \rightarrow\cdots\rightarrow F_2 \xrightarrow{\alpha_2} F_1\xrightarrow{\alpha_1} F_0 \to F_{-1} \to F_{-2}.
\end{equation}
By \Cref{thmBE73}, this condition implies that $\grade \alpha_m \ge m + 2$ for $i = 1,\dots, m$.
\end{proof}
 
 \begin{proof}[Proof of \Cref{thmRk2}]
 Now the proof follows from \Cref{lemRk2} and its proof. 
\end{proof}

A similar argument as in the proof of \Cref{lemRk2}, one can show that the exactness of the complex in \Cref{thmRk1} follows from the condition that $\rk \Sym(I)_{\delta-i}=1$ and $\Sym(I)_{\delta-i}$ satisfies Serre's condition $(S_1)$. Hence it is natural to ask the following question. 

\begin{question}\label{rge3Question}
Consider the complex 
	\begin{equation}
	\wedge^{f_0-r_1-1}F_0\otimes S(-s_1)\xrightarrow{\eta\circ m_\wedge \circ(\ide\otimes a_1)}F_0^*\xrightarrow{\alpha_1^*}F_1^*
	\end{equation}
in \Cref{thm1}. 
Assume that $\rk \Sym(I)_{\delta-i}= r > 0$. Then is
this complex exact if $\Sym(I)_{\delta-i}$ satisfies Serre's condition $(S_r)$?
Furthermore, in the case where this complex is exact, is there a natural way to construct the minimal free resolution of $\coker \alpha_1^*$? 
\end{question}

This question has an affirmative answer when $r \le 2$ by \Cref{thmRk1,thmRk2}. 

\section{Degrees of Implicit equation}\label{sectionOnDegreesofImplicit}
In this section, we provide a concrete description of the bi-degree of the equation defining the special fiber ring $\mathcal{F}(I)$. With the help of Morley forms, the results in this section can be used to find the equation defining $\mathcal{F}(I)$. 
We assume the notation and set up of \Cref{setupProp}, but for the convenience of the reader we recall the key notation and setup.  By $d_1, \dots, d_n$ we denote the column degrees of the graded presentation matrix $\varphi$ of $I$, $n$ is the dimension of $R$, $I = I_n(\varphi)$, and $\delta=d_1+\cdots+ d_n-n = d-n$. 
Consider the diagram, where the rows are exact
\begin{equation}
\xymatrix{
0 \ar[r] & \mathcal{A}\ar[r]\ar[d] & \Sym(I)\ar[r]\ar[d] & R[It]\ar[r]\ar[d] & 0\\
0\ar[r] & \ker\overline{\pi}\ar[r] & S=\kk[T_0,\dots,T_n] \ar[r]^{\hspace{1cm}\overline{\pi}}& \mathcal{F}(I)\ar[r] & 0.
}
\end{equation}
The defining ideal of $\mathcal{F}(I)$ is $\ker\overline{\pi}\cong\mathcal{A}\otimes\kk\cong \mathcal{A}_0 S$.  
By \Cref{perfPair}, to find the defining ideal of $\mathcal{F}(I)$, it is enough to compute $\Hom_S(\Sym_R(I)_\delta,S(-n))$.\\

We begin with the resolution of $\Sym(I)_{\delta}$, which is induced from the resolution in \Cref{pieceRes}.
\begin{align}\label{symdeltares}
0\rightarrow F_m\xrightarrow{\alpha_m}F_{m-1}\rightarrow\cdots\rightarrow F_2\xrightarrow{\alpha_2}F_1\xrightarrow{\alpha_1}F_0\rightarrow\Sym(I)_\delta\rightarrow 0,
\end{align}
where $m\leq n-1$ and since $\delta + n = d = d_1 + \cdots + d_n$, we may write
\begin{align}
F_t=\bigoplus_{1\leq j_1\leq j_2\leq\cdots\leq j_t\leq n}S(-t)^{\delta-(d_{j_1}+\cdots+d_{j_t})+n-1\choose n-1}=\bigoplus_{1\leq k_1\leq k_2\leq\cdots\leq k_{n-t}\leq n}S(-t)^{d_{k_1}+\cdots+d_{k_{n-t}}-1\choose n-1},
\end{align}
where $\{ j_1, ,\dots, j_t, k_1,\dots,k_{n-t} \} = \{1,\dots,n\}$. 
As in the previous section, we let $r_t=\rk~\alpha_t, f_t=\rk ~F_t$ 
and $\mathcal{A}_0\cong \Hom(\Sym(I)_\delta,S(-n))=\ker\alpha_1^*(-n)$.

\begin{remark}\label{remarkonextraranks}
The length of the Koszul complex in \Cref{pieceRes} is $n$ whereas the length of the resolution of $\Sym(I)_\delta$ in \cref{symdeltares} can be strictly less than $n-1$.  
In particular, $F_t = 0$ for $m+1 \le t \le n-1$. 
In other words, we have
\begin{equation}
{d_{k_1}+\cdots+d_{k_{n-t}}-1\choose n-1} = 0~\iffw~ m+1\le t. 
\end{equation}
\end{remark}

By \cite[Proposition 2.4]{UV}, $\ker\overline{\pi}$ is a principal prime ideal  and hence, so is $\ker\alpha_1^*$.
By  Theorem \ref{thm1}, we have a complex 
\begin{equation}
\wedge^{f_0-r_1-1}F_1\otimes S(-s_1)\xrightarrow{\eta\circ m_\wedge \circ (\ide\otimes a_1)}F_0^*\xrightarrow{\alpha_1^*}F_1^*,
\end{equation}
 where $\eta$ is a fixed orientation of $F_0$ and $a_1:S(-s_1)\rightarrow\wedge^{r_1}F_0$ is the map obtained by applying  \Cref{BEprelim} to the resolution \cref{symdeltares}. 
 We refer to \Cref{reviewBE74}(2) for the description of the shifts $s_1$ in the preceding complex. Since $\rk \Sym(I)_\delta=1$, we have $f_0-r_1-1=0$, and the above complex simplifies to  $S(-s_1)\xrightarrow{\eta(a_1)}F_0^*\xrightarrow{\alpha_1^*}F_1^*$. 
 Thus $\eta(a_1(1_{S(-s_1)}))\in \ker\alpha_1^*$. Henceforth, we will often denote the basis element $1_{S(-s_1)}$ of $S(-s_1)$ simply by $1$ if no confusion arises.  \\

Our main result of this section is the following theorem. 

\begin{theorem}\label{thm5}Assume $\kk$ is a field of characteristic zero and the notation and hypothesis of \Cref{setupProp}. Let $R=\kk[x_1,\dots,x_n]$ with $n \ge 3$, and let $\phi : \Proj R=\PP^{n-1} {\xdashrightarrow{[f_0:\cdots: f_n]}}  \mathbb{P}^{n}_\kk$ be the rational map defined in \cref{parametrizationmap}. Write $I = (f_0,\dots, f_n)$. Then the following statements are equivalent:
	\begin{enumerate}[\indent$(1)$]
		\item The rational map $\phi$ is birational onto its image.	
		\item The $\ker\alpha_1^*$ is generated by $(\eta \circ a_1)(1)$.
		\item $\grade I(\eta(a_1))\geq 2$.
		\item The greatest common divisor of the entries of $\eta \circ a_1$ (this is a column matrix) is 1.
		\item $e(\mathcal{F}(I))= e ({R}/{(g_0,\dots,g_{n-2}):I})$, where $g_0, \dots, g_{n-2}$ are general $\kk$-linear combinations of the $f_i'$s. Here, $e(-)$ denotes, the Hilbert-Samuel multiplicity. 
	\end{enumerate}
\end{theorem}

The Hilbert-Samuel multiplicity $e(\mathcal{F}(I))$ is also called the \textit{degree} of the projective variety $\Proj \mathcal{F}{(I)} = \overline{\Phi (\mathbb{P}^{n-1})}$. 
We use the next two lemmas to prove \Cref{thm5}.

\begin{lemma}[{\cite[Corollary 3.7]{KPU16}}]\label{KPUmultiplicity}
	Let $\kk$ be a field of characteristic zero, $R=\kk[x_1,\dots,x_n]$ with $n \ge 3$. Further, let $I$ be a homogeneous ideal generated by forms of the same degree $d$ and $\kk[R_d]$ the $d$-th Veronese subring of $R$. 
	If $g_0,\dots,g_{n-2}$ are general $\kk$-linear combinations of the generators $f_0,\dots,f_n$ of $I$, then the following equality holds.
	\begin{equation}\label{KPUmultiplicityeq}
	e(\mathcal{F}(I))=\frac{1}{[\kk[R_d]:\kk[f_0,\dots,f_n]]}\cdot e\left(\frac{R}{(g_0,\dots,g_{n-2}):I^\infty} \right).
	\end{equation}
\end{lemma}
In the above theorem $[\kk[R_d]:\kk[f_0,\dots,f_n]]$ denotes the field extension degree $[\Quot(\kk[R_d]):\Quot(\kk[f_0,\dots,f_n])]$. Using \Cref{birationalprelim}, we have that the rational map $\phi$ (as defined in \cref{parametrizationmap}) is birational onto its image if and only if $[\kk[R_d]:\kk[f_0,\dots,f_n]]=1$. 

\begin{defn}
	For any matrix $A$ with entries of the same degree, we define $\deg~ A$ to be the degree of the entries of $A$.
\end{defn}

\begin{lemma}\label{thm4} Assume the notation and hypothesis of \Cref{thm5}. 
For general $\kk$-linear combinations $g_0,\dots,g_{n-2}$ of the generators $f_0,\dots,f_n$ of $I$, we have
\begin{equation}\label{deg of entries and multiplicity}
\deg \operatorname{Im} (\eta \circ a_1)+\dim R=e\left(\frac{R}{(g_0,\dots,g_{n-2}):I} \right) = e\left(\frac{R}{(g_0,\dots,g_{n-2}):I^\infty} \right).
\end{equation}
\end{lemma}

\begin{proof}
First we show that for general linear combinations $g_0,\dots,g_{n-2}$, 
\begin{equation}
e({R}/{(g_0,\dots,g_{n-2}):I^\infty} ) = e({R}/{(g_0,\dots,g_{n-2}):I} ).
\end{equation}

For each $\lambda=(\lambda_{0},\cdots,\lambda_{n}) \in (\mathbb{A}_\kk^{n+1})^{n+1}$, where $\lambda_i=(\lambda_{i0}:\cdots:\lambda_{in}) \in\mathbb{A}^{n+1}_\kk$ we set $g_{i}=\sum_{j=0}^n\lambda_{ij}f_j\in R$ for $0\leq i\leq n$. 
Let $U'$ be an open subset of $(\mathbb{A}_\kk^{n+1})^{n+1}$ such that $\det (a_{ij})\in\kk\backslash\{ 0\}$. 
This shows that $\{g_0,\dots,g_{n}\}$ is a minimal generating set of $I$ if $\lambda \in U'$. Further, let $U_1'=U'\cap (\mathbb{A}_\kk^{n+1})^{n-1}$ (the first $n-1$ components of $(\mathbb{A}_\kk^{n+1})^{n+1}$).\\

Notice that the ideal $I$ satisfies $\mu(I_p)\leq \dim R_p$ for every $p\in \spc(R)\backslash\{\m\}$. 
By \cite[Corollary  1.6, Proposition 1.7]{Ul1} (see also \cite[Lemma 3.1]{PoliniXie}), there exists a dense open set $U''\subseteq (\mathbb{A}_\kk^{n+1})^{n-1}$ such that for every $\lambda \in U''$, $(g_0,\dots,g_{n-2}):I$ is a geometric $(n-1)$-residual intersection. 
(We note that in \cite{PoliniXie}, the non-empty open subset $U''$ corresponds to general $R$-linear combinations. 
Thus, the same $U''$ can be used as the desired non-empty open subset for $\kk$-linear combinations.)  \\

It follows that $(g_0,\dots,g_{n-2})_P=I_P$ for $P\in V(I)$ such that $\hgt P\leq n-1$. Since $I$ is a codimension two perfect ideal, $I$ is a strongly Cohen-Macaulay ideal. Now we apply \cite[Theorem 3.1]{Hun1}   to conclude that $(g_0,\dots,g_{n-2}):I^\infty=(g_0,\dots,g_{n-2}):I$.
The open set $U = U'_1 \cap U''$ provides the desired non-empty open subset.\\

As $g_0,\dots, g_{n}$ form a minimal homogeneous generators of $I$, we may choose a presentation matrix $\varphi'$ for $(g_0,\dots, g_n)$ with the same column degrees $d_1, \dots, d_n$ of $\varphi$. This follows from the uniqueness of the graded Betti numbers of $I$. Furthermore, we assume that $g_i$ are signed minors of the $(n+1) \times n$ matrix $\varphi'$.  \\

Let $M=I/(g_0,\dots,g_{n-2})$. Notice that the last two rows of $\varphi'$ is a presentation matrix for $M$, and we call this presentation matrix by $\varphi''$. We note that $\varphi''$ is a $2 \times n$ matrix. Furthermore, $I_2 (\varphi'') = {\fitt_0(M)}={\ann M}={(g_0,\dots,g_{n-2}):I}$ (cf. \cite[Theorem 3.1(2)]{BE77}) has the maximum possible height $n-1$. 
Thus we reduce to the case of computing the multiplicity  $e({R}/{I_2(\varphi'')} )$. 
Since the $\grade I_2(\varphi'')$ is the maximum possible number, $n-1$, the Eagon-Northcott complex of $\varphi''$, $\operatorname{EN}(\varphi'')$, forms a graded (minimal) free $R$-resolution for $R/I_2(\varphi'')$.
Let $N=R/I_2(\varphi'')$. Clearly, $e(N)=e(N(1))$, and the shifted Eagon-Northcott complex of $\varphi''$ is the resolution of $N(1)$.
\begin{multline}
\operatorname{EN}(\varphi'')(1): \\
0 \to (\Sym(R^2)_{n-2})^* \otimes R^{n-1}(-(d_1+\cdots d_n-1)) \to (\Sym(R^2)_{n-3})^* \otimes \bigoplus_{1\leq j_1\leq\cdots\leq j_{n-1}\leq n}R^{n-2}(-(d_{j_1}+\cdots+d_{j_{n-1}}-1)) \\ 
\to \cdots \to (\Sym(R^2)_0)^* \otimes \bigoplus_{1\leq j_1\leq j_2\leq n}R(-(d_{j_1}+d_{j_2}-1)) \to R(1)\rightarrow 0.
\end{multline}
Furthermore, the Hilbert Series of $N(1)$ is
\begin{equation}
\frac{\sum\limits_{t=1}^{n-1}\left[(-1)^t\cdot t\sum\limits_{1\leq j_1\leq\cdots\leq j_{t+1}\leq n} z^{(d_{j_1}+\cdots+d_{j_{t+1}}-1)}\right]+z^{-1}}{(1-z)^n}.
\end{equation}
Let $p(z)={\sum\limits_{t=1}^{n-1}\left[(-1)^t\cdot t\sum\limits_{1\leq j_1\leq\cdots\leq j_{t+1}\leq n} z^{(d_{j_1}+\cdots+d_{j_{t+1}}-1)}\right]+z^{-1}}$. Then 
\begin{align}
e(N(1)) & = (-1)^{n-1}\frac{p^{(n-1)}(1)}{(n-1)!}\\
& = (-1)^{n-1}\left[\sum_{t=1} ^{n-1}(-1)^t\cdot t\cdot \sum_{1\leq j_1\leq\cdots\leq j_{t+1}\leq n}{d_{j_1}+\cdots + d_{j_{t+1}}-1\choose n-1} \right] + (-1)^{n-1}\\
& = \sum_{t=1} ^{n-1}\left[(-1)^{n-1+t}\cdot t\cdot \sum_{1\leq j_1\leq\cdots\leq j_{t+1}\leq n}{d_{j_1}+\cdots + d_{j_{t+1}}-1\choose n-1}\right] + 1, 
\end{align}
where $p^{(n-1)}(z)$ denotes the $(n-1)^{\text{st}}$ derivative of $p(z)$. 
We compare the terms with the ranks in \cref{symdeltares}. 
By \Cref{remarkonextraranks}, we have the following equality.
\begin{equation}
\sum_{1\leq j_1\leq\cdots\leq j_{t+1}\leq n}{d_{j_1}+\cdots + d_{j_{t+1}}-1\choose n-1}=
\begin{cases} 
\operatorname{rk} F_{n-t-1} = f_{n-t-1} & t\geq n-m-1, \\
0 & t<n-m-1.
\end{cases}
\end{equation}

Under this identification, we may rewrite $e(N(1))$ as follows
\begin{align}
e(N(1)) & =\sum_{t=n-m-1} ^{n-1}\left[(-1)^{n-1-t}\cdot t\cdot f_{n-t-1}\right] + 1\\
& = (-1)^m(n-m-1)f_m+(-1)^{m-1}(n-m)f_{m-1}+(-1)^{m-2}(n-m+1)f_{m-2}+\cdots +(n-1)f_0 + 1.
\end{align}
We note that since $\rk \Sym(I)_\delta=1$, $\sum_{j=0}^m(-1)^jf_j=1$.
Then one has 
\begin{align}
n-m - 1 &= (n-m - 1) \sum_{j=0}^m(-1)^jf_j \\
&= (-1)^m (n-m-1) f_m + (-1)^{m-1}(n-m-1) f_{m-1} + (-1)^{m-2}(n-m-1) f_{m-2} + \cdots + (n-m-1) f_0 \\
&= (-1)^m (n-m-1) f_m + (-1)^{m-1}(n-m) f_{m-1} +  (-1)^{m-2}(n-m+1) f_{m-2} + \cdots + (n-1) f_0 \\
&\phantom{(-1)^m (n-m-1) f_m +} - [(-1)^{m-1}(1) f_{m-1} +  (-1)^{m-2}(2) f_{m-2}+\cdots + (m) f_0] \\
&= [e(N(1)) - 1] - [(-1)^{m-1}(1) f_{m-1} +  (-1)^{m-2}(2) f_{m-2}+\cdots + (m) f_0] .
\end{align}
Hence 
\begin{align}
e(N(1)) &= (-1)^{m-1}(1) f_{m-1} +  (-1)^{m-2}(2) f_{m-2}+\cdots + (m) f_0 + n - m  \\
 &= \sum_{k=1}^{m}(-1)^{m-k}\cdot k\cdot f_{m-k} + n - m . 
\end{align}
Thus, we have shown that  
\begin{equation}\label{multplicity of saturation}
e({R}/{(g_0,\dots,g_{n-2}):I^\infty} ) = \sum_{k=1}^{m}(-1)^{m-k}\cdot k\cdot f_{m-k}+(n-m).
\end{equation}
We will complete the proof by showing that the right hand side of the equation is equal to $\deg \operatorname{Im} (\eta \circ a_1)+n$. First observe that 
 $\deg \operatorname{Im} (\eta \circ a_1)+n =\deg (a_1)+n$  since $\eta$ is an isomorphism of degree 0.
In \Cref{reviewBE74}, we showed that the degree of $\operatorname{Im} {a_1}$ is $s_1 = \sum_{t=1}^{m}(-1)^{t-1}r_t$,
and $r_t=\sum_{j=t}^{m}(-1)^{j-t}f_j$.
Note that since $\sum_{j=0}^m(-1)^jf_j=1$, we have 
\begin{equation}
1 = f_0 - f_1 + \cdots + (-1)^{t-1} f_{t-1} + \underbrace{(-1)^t f_t + (-1)^{t+1} f_{t+1} + \cdots + (-1)^m f_m}_{(-1)^t r_t}.
\end{equation}
Therefore, 
\begin{equation}
r_t = (-1)^{t+1} (f_0 - f_1 + \cdots (-1)^{t-1} f_{t-1}) + (-1)^t = \sum_{j=0}^{t-1}(-1)^{t+j+1}f_j+(-1)^t.
\end{equation}
Finally, we have 
\begin{align}
\deg \operatorname{Im} (\eta \circ a_1)+n  &=\deg (a_1)+n\nonumber \\
&= \sum_{t=1}^{m}(-1)^{t-1}r_t +n\nonumber\\
& = \sum_{t=1}^m (-1)^{t-1}\left(\sum_{j=0}^{t-1}(-1)^{t+j+1}f_j+(-1)^t \right)+n\nonumber\\
& = \sum_{t=1}^m \left(\sum_{j=0}^{t-1} (-1)^j f_j-1\right) +n\nonumber\\
& = \sum_{t=1}^m \sum_{j=0}^{t-1} (-1)^j f_j+\sum_{t=1}^m (-1) +n \\
&= \sum_{t=1}^m \sum_{j=0}^{t-1} (-1)^j f_j+m(-1) +n\nonumber\\
& = \left(f_0+(f_0-f_1)+(f_0-f_1+f_2)+\cdots +\left(\sum_{j=0}^{m-1}(-1)^jf_j\right)\right)+(n-m)\nonumber\\
&= \sum_{k=1}^{m}(-1)^{m-k}\cdot k\cdot f_{m-k}+(n-m).
\end{align}
This completes the proof. 
\end{proof}

\begin{proof}[Proof of \Cref{thm5}] 
We note that the rational map $\phi$ is birational to the image if and only if $[\kk[R_d]:\kk[f_0,\dots,f_n]] = 1$. \\
$(1)\Leftrightarrow (5)$: This follows from \Cref{KPUmultiplicity,thm4}.

\noindent $(2)\Leftrightarrow (5)$: As $\mathcal{F}(I)$ is a hypersurface, $e(\mathcal{F}(I))$ is the degree of a generating element of $\ker \overline{\pi}$. The equivalence follows from \Cref{KPUmultiplicity,thm4} which say
\begin{equation}
	e(\mathcal{F}(I))=\frac{\deg \operatorname{Im} (\eta \circ a_1)+n}{[\kk[R_d]:\kk[f_0,\dots,f_n]]}.
\end{equation}

\noindent $(2)\Leftrightarrow(3)$: This is a consequence of \Cref{thmRk1}.

\noindent $(3)\Leftrightarrow (4)$: This equivalence is clear as the ambient ring is a UFD. 
\end{proof}

\Cref{thm1Example}(b) is an example where $I$ satisfies the conditions in \Cref{setupProp}, but the rational map is not birational to its image. In this example $\deg a_1 + n = 20$, but $e(\mathcal{F}(I) ) = 10$, so that the extension degree is $2$. 

\section{The Main result and application}\label{mainResultsec}
Let $R = \kk[x,y,z]$ be a polynomial ring over a characteristic zero field $\kk$ and $M$ a $4$ by $3$ matrix with entries in $R$.  
Our main result states that if $M$ is general of type $(1, 2, \nn)$ for some $\nn > 2$ (see Case 2 in \Cref{defGenMat} for the notion of general), then $I = I_3(M)$ satisfies the equivalence conditions in \Cref{thmRk1,thmRk2}. 
Therefore, we are able to state the bi-degrees and the minimal number of equations in each bi-degree of the defining ideal of the Rees algebra $R[It]$, explicitly (hence for $\mathcal{F}(I)$ as well). Furthermore, by the duality of Jouanolou (\Cref{perfPair}), one may recover the explicit equations in $\Sym_R(I)$.

\begin{remark}\label{rmkMatBlock}
Let $B = S[x_1,\dots, x_d]$, where $S$ is a ring, and $f(x_1,\dots,x_d)$ a non-zero homogeneous polynomial of degree $\nn$ in $B$. The graded homomorphism
$\Phi: B(-\nn) \stackrel{\cdot f}{\to} B$ induces a homomorphism in each degree $i$. We call the induced map $\Phi^f_i :  [B(-\nn)]_i \to B_i$. This map can be described explicitly once we fix a basis of $B_i$. 
In this note, we will always use the Lex (monomial) order on $x_1,\dots, x_d$. For instance, when $d = 3$, $\nn = 2$, $f = A_0 x_1^2 + A_1 x_1x_2 + A_2x_3^2$, where $A_i \in S$, then $\Phi^f_3$ has a matrix representation
\begin{equation}
\begin{bmatrix}
A_0    &0       	&0 	\\
A_1  & A_0	&0	\\
0        &0 		& A_0 \\	
0        & A_1	&0	\\
0           &0		& A_1 \\
 A_2   &0		&0 \\
0          &0		&0\\
0         &0 		&0\\
0         & A_2 &0\\
0          &0		& A_2
\end{bmatrix}
\,
\overbrace{
\begin{array}{:l}
{x_1}^3  \\
{x_1}^2{x_2}  \\
{x_1}^2{x_3} \\
{x_1}{x_2}^2 \\
{x_1}{x_2}{x_3} \\
{x_1}{x_3}^2 \\
{x_2}^3 \\
{x_2}^2{x_3} \\
{x_2}{x_3}^2 \\
{x_3}^3
\end{array}
}^{basis} 
.
\end{equation}
Here each column is the presentation of $x_i f$ by the monomials of degree $3$.
\end{remark}

To prove our main theorem, we need to compute the height of the maximal minors of $\Phi^f_i$.  
The following lemma generalizes \cite[Theorem A.2.60]{Eis05} in which there is no gap between nonzero entries. 

\begin{lemma}\label{lemStaircase}
Let $R$ be a ring and $g_1, \dots, g_t$ are elements of $R$. Let $M=(m_{ij})$ be a $r\times s$ matrix  ($r\geq s$) satsifying 
\begin{enumerate}[\indent $(1)$]
	\item each entry of $M$ is either zero or $g_1,\dots,g_t$, 
	\item the $i$-th non zero entry of each column is $g_i$, 
	\item the last non-zero entry of each column is $g_t$, and 
	\item if $m_{ij}=g_q$, then $m_{k j+1}=g_q$ for some $k>i$.
\end{enumerate}
	Then $I_s(M)=(g_1,\dots,g_t)^s$.
\end{lemma}
	
\begin{proof}
It suffices to show the statement that if $R' = \mathbb{Z}[y_1\cdots y_t]$ and $M=(m_{ij})$ a $r\times s$ matrix  ($r\geq s$) satisfying
	\begin{enumerate}
		\item each entry of $M$ is either zero or $y_1,\dots,y_t$, 
		\item the $i$-th non zero entry of each column is $y_i$,
		\item the last non-zero entry of each column is $y_t$, and 
		\item if $m_{ij}=y_q$, then $m_{k j+1}=y_q$ for some $k>i$,
	\end{enumerate}
then $I_s(M)=(y_1,\dots,y_t)^s$.

Clearly, $ I_s(M)\subseteq (y_1,\dots,y_t)^s$. We now show the other inclusion.  
We prove by induction on the number of variables $t$ and the number of columns $s$ of the matrix. Suppose $t=1$. The case of $s=1$ is clear, by the pattern of the matrix, that $I_t(M)=(y_1)$. 
In fact, if $s>1$, then it is easy to see from item (4) that $I_s(M)=(y_1)^s$ settling the base case of induction. Now suppose by induction hypothesis, the result is true for all matrices $M$ following the pattern in the hypothesis, with entries in the ring $T=\mathbb{Z}[y_1,\dots,y_q]$ where $q\leq t-1$ and for all $s$.\\

If $s=1$, in which case $M$ has only one column, then clearly  $I_s(M)=(y_1,\dots,y_t)$. By induction hypothesis on $s$ assume that the result is true for all matrices $M$, satisfying the pattern in the hypothesis, with the number of columns being strictly less than $s$. Now let  $M$ be a matrix with $s$ columns and $j_1,\dots,j_s$ be indices such that $m_{1j_1}=m_{2j_2}=\cdots=m_{sj_s}=y_1$. If $M'$ denotes the matrix obtained by deleting the $j_1$-th row and the first column of $M$, then by induction hypothesis on $s$, it is clear that $I_{s-1}(M')=(y_1,\dots,y_t)^{s-1}$. 
Since the $j_1$-th row of $M$ has $y_1$ appearing in the first entry and zero otherwise, we have $f=y_1\cdot g$ where $g\in I_{s-1}(M')=(y_1,\cdots,y_t)^{s-1}$. 
Thus one has $y_1(y_1,\dots,y_t)^{s-1}\subseteq I_s(M)$. We show similarly that $y_1^u(y_1,\dots,y_t)^{s-u}\subseteq I_s(M)$. Let $M_u$ be the matrix obtained by removing the rows $j_1,\dots,j_u$ and the columns $1,\dots,u$ from $M$. Then by induction hypothesis $I_{s-u}(M_u)=(y_1,\dots,y_t)^{s-u}$. Now consider $M$ and an $s$-minor $f'$ obtained by choosing the rows $j_1,\dots,j_u,i_1,\dots,i_{s-u}$ where $1\leq i_q\leq r, i_q\not\in\{j_1,\dots,j_{u}\}$ for $1\leq q\leq s-u$, of $M$. Then we have that $f'=y_1^ug'$ where $g'\in I_{s-u}(M_u)=(y_1,\dots,y_t)^{s-u}$. Thus we have $y_1^u(y_1,\dots,y_t)^{s-u}\subseteq I_s(M)$ for $1\leq u\leq s$. \\

Now, consider the ring $\overline{T}=T/(y_1)$ and the matrix $\overline{M}$ obtained by extending the matrix $M$ to  $\overline{T}$. $\overline{M}$ follows the pattern in the hypothesis and hence by induction hypothesis on $t$, we have $I_{s}(\overline{M})=(\overline{y_2},\cdots,\overline{y_t})^{s}$. 
Thus we have $I_s(M) + (y_1) \supseteq (y_2,\dots,y_t)^s$. It suffices to show that for each monominal generator $m$ for $(y_2,\dots,y_t)^s$ is in $I_s(M)$. This follows immediately as $y_i$ are variables. Thus, we have $(y_1,\dots,y_t)^s\subseteq I_s(M)$. 
\end{proof}

\begin{thm}\label{mainThm}
Let $R = \kk[x,y,z]$, where $\kk$ is a field of characteristic $0$. 
Let $M$ be a $4 \times 3$ matrix whose entries are in $R$ and let $I = I_3(M)$. 
If $M$ is general of type $(1,2, \nn)$, where $\nn > 2$ (see Case 2 in \Cref{defGenMat}),
then the defining ideal of the Rees algebra $R[It]$ is minimally generated by 
\begin{equation}
\def\arraystretch{1.2}
\begin{array}{c;{2pt/2pt}c|c}
\text{~} & \text{bidegree} & \text{number of elements}  \\
\hline
l_1 & (1,1) & 1 \\ \hline
l_2 & (2,1) & 1 \\ \hline
l_3 & (\nn,1) & 1 \\ \hline
\mathcal{A}_0 & (0, 3\nn+2 ) &  1 \\ \hline
\mathcal{A}_{\nn - i} &  (\nn-i, 3i+1 )&   {2 + i \choose 2}  \\ \hline
\mathcal{A}_\nn & (\nn, 3) & 1  

\end{array}
\end{equation}
 for $1 \le i \le \nn-1$.  
In particular, the defining ideal is minimally generated by ${\nn+2 \choose 3} + 4$ elements. 
\end{thm}

\begin{proof}
The general property $\mathcal{P}$ needs to satisfy the following conditions. 
\begin{enumerate}
\item $I_3(M)$ is of height $2$, and $I$ satisfies $(G_3)$, i.e, $\hgt I_3(M) \ge 2$ and $\hgt I_2(M) \ge 3$. 

\item For all $1 \le i \le \nn$, $\rk \Sym_i = 1$ or $2$, and  the complex in \Cref{thmRk1} or \Cref{thmRk2} is exact. (This exactness can be verified by checking the height of $I (\alpha_j)$, where $\alpha_j$ are the differentials in the complex by \Cref{thmBE73}.)
\end{enumerate}
Notice that we have finitely many conditions. 
If each condition is general corresponding to a non-empty open subset $U_k$, 
then the desired non-empty subset for $\mathcal{P}$ is the intersection of the open subsets $U_k$. 
To show that each condition is general, first we apply \Cref{corGeneralHeight}(b) to show the existence of such open subsets.
To finish the proof, we need to show that each open subset is not empty. 
This can be done by demonstrating an example satisfying the conditions in (1) and (2). 
We do this in \Cref{bigExample}.
The generating degrees follow from \Cref{thmRk1,thmRk2}. 
This completes the proof.
\end{proof}

\begin{example}\label{bigExample}
Let $R = \kk[x,y,z]$, where $\kk$ is a field of characteristic $0$, $\m = (x,y,z)R$, and $M$ the matrix
\begin{equation}
M = \begin{bmatrix}
x & y^2 & \gamma z^{\nn} \\
y & z^2 & yz^{\nn-1}\\
z & x^2 & y^{\nn}  \\
0 &xy + yz +xz  & 0 
\end{bmatrix},
\end{equation}
where $\gamma \in \kk \setminus \{0\}$. We will show that there exists a nonzero $\gamma$ such that the ideal $I = I_3(M)$ satisfies the desired properties required in the proof of \Cref{mainThm}.\\

\noindent\textbf{Claim 1:} $I$ is a height two perfect ideal satisfying the $G_3$ condition. 
\begin{proof}[Proof of Claim 1]
By Krull's principal ideal theorem, it suffices to show that $(I,x)$ is of height three. (We will use this technique multiple times in the proof.) Equivalently, we show that height $I (R/xR)$ is two. By abuse of notation, we use the same symbols for $y,z$ for their images in $\overline{R} := R/xR$. Let 
\begin{equation}
\overline{M} = \begin{bmatrix}
0 & y^2 & \gamma z^\nn \\
y & z^2 &  yz^{\nn-1}\\
z &  0 & y^\nn  \\
0 & yz   & 0 
\end{bmatrix}.
\end{equation}
Notice that from the determinant after deleting the second row of $\overline{M}$ is $\gamma y z^{\nn+2}$, and 
the determinant after deleting the last row of $\overline{M}$ is $y^{\nn+3} - y^3 z^{\nn} + \gamma z^{\nn+3}$. 
Hence any prime containing $I_3(\overline{M})$ contains $(y,z)$. 
Since $I_3(M) \overline{R}= I_3 (\overline{M})$, this shows that $\hgt I \ge 2$. Then by the Hilbert-Burch theorem \cite[Theorem 20.15]{Eis}, $\hgt I = 2$. \\

For the $G_3$ condition, it suffices to show that $\hgt F_2 (M) = I_2 (M) \ge 3$ (\Cref{Gs}). 
We will show that any prime ideal containing $I_2(M)$ is of height $3$ (indeed, it will turn out that $(x,y,z)$ is the only prime containing $I_2(M)$). 
Let $\p$ be a prime ideal containing $I_2(M)$. 
From the minor (of rows 2,3 and columns 1,3 of $M$) $\begin{vmatrix} y & yz^{\nn-1} \\ z & y^{\nn} \end{vmatrix} = y^{\nn+1} - yz^{\nn} = y^{\nn} (y-z)$, 
we see that $y$ or $y-z$ is in $\p$.
If $y \in \p$, then from the minors $yx^2-z^3,  x^3-zy^2$ (from the first two columns of $M$), we see that $x,z$ are also in $\p$. 
So, $\p = (x,y,z)$. 
If $y-z \in \p$, then from the minor $y^3-xz^2$, we have $y^3 - xy^2 = y^2(x-y)$ is in $\p$. 
If $y$ is in $\p$, then we are done by the previous argument. 
So, we assume that $x-y \in \p$. 
Then from the minor $x ( xy+yz+xz)$, we see that $x(3x^2) = 3x^3$ is in $\p$ as $x-y,y-z \in \p$.  
Hence $x, y-x, z-x$ is in $\p$, and this implies $\p = (x,y,z)$.
Thus, we have $\hgt I_2(M) = 3$, and this completes the proof of claim 1.
\end{proof}

\noindent\textbf{Claim 2:} Let $S = \kk[T_0,T_1,T_2,T_3], B = R \otimes_{\kk} S = S[x,y,z]$, and $\delta = 1 + 2 + \nn -3 = \nn$. Then we have the following.

\begin{enumerate}[\indent$(a)$]
\item $[\Sym_R(I)]_0$ is $S$-free of rank $1$, and 
for $1 \le i \le 2$, the graded $S$-resolutions of $[\Sym_R(I)]_i$ are 
\begin{equation}
0 \to S(-1) \stackrel{\alpha_1}{\to} S^{3} \to [\Sym_R(I)]_1 \to 0
\end{equation}
and
\begin{equation}\label{eqSym1}
0 \to S^{3}(-1) \oplus S(-1)  \stackrel{\alpha_1}{\to} S^{6} \to [\Sym_R(I)]_2 \to 0,
\end{equation}
respectively, and $\hgt I(\alpha_1) = 3$.

\item For $3 \le i \le \delta -1$, the graded $S$-resolution of $[\Sym_R(I)]_i$ is 
\begin{equation}\label{midRes}
0 \to S^{2+i-3 \choose 2}(-2) \stackrel{\alpha_2}\to S^{2+ i-1 \choose 2}(-1) \oplus S^{2+i-2 \choose 2}(-1)  \stackrel{\alpha_1}{\to} S^{2 + i  \choose 2} \to [\Sym_R(I)]_i \to 0,
\end{equation}

and $\hgt I(\alpha_1) = 3$ and $\hgt I(\alpha_2) = 4$.

\item The graded $S$-resolution of $[\Sym_R(I)]_\delta=[\Sym_R(I)]_{\nn}$ is 
\begin{equation}
0 \to S^{2+\nn-3 \choose 2}(-2)  \stackrel{\alpha_2}\to S^{2+ \nn-1 \choose 2}(-1) \oplus S^{2+\nn-2 \choose 2}(-1) \oplus S(-1)  \stackrel{\alpha_1}\to  S^{2 + \nn  \choose 2} \to[\Sym_R(I)]_\nn \to 0,
\end{equation}
and $\hgt I(\alpha_1) = 2$. 
\end{enumerate}

\begin{proof}[Proof of claim 2] 
The shape of the complexes follows from \Cref{pieceRes}. 
The rank of each free module in these complexes can be calculated easily as they are the number of monomials in $(x,y,z)$ of degree $i$,
and the rank of $[\Sym_R(I)]_i = 1$ if $i = 0,1,\delta$ and $2$ otherwise.
By \Cref{thmRk1,thmRk2}, to check the exactness of these complexes, it suffices to verify the claimed heights of $I(\alpha_j)$ (\Cref{thmBE73}). \\

We will use the notation $[ \, l_0 \, l_1 \, l_2] = [\, T_0 \, \dots \, T_3] \, M$. That is 
\begin{align}
l_0 &= xT_0+y T_1+z T_2,\\
l_1 &= y^2T_0 +z^2T_1 + x^2T_2+(xy+yz+zx) T_3,\\
l_2 &= \gamma z^{\nn} T_0 + yz^{\nn-1} T_1 + y^{\nn}T_2.
\end{align}
 First, we show that $I(\alpha_2) \ge 4$ for the cases (b) and (c). By \Cref{rmkMatBlock}, the matrix representation of $\alpha_2$ is 
\begin{equation}
\def\arraystretch{1.8}
\left[
    \begin{array}{c}
        \Phi^{-l_1}_{i} \\ \hdashline[2pt/2pt] 
        \Phi^{l_0}_{i}
    \end{array}
\right].
\end{equation}
Each column consists of exactly 7 non-zero entries including $T_0,T_1,T_2,T_3$, and it satisfies the conditions of \Cref{lemStaircase}. Thus, we conclude that $\hgt I(\alpha_2) = 4$.\\

When $i = 1$, $\alpha_1$ involves only $l_0$, and the map $\alpha_1 = \Phi^{l_0}_1$ is the column matrix consists of $T_0,T_1,T_2$. Hence $\hgt I(\alpha_1) = 3$. 
Notice that for $2 \le i \le \delta -1$ the free $S$-resolution in \cref{eqSym1,midRes} involves only the elements $l_0,l_1$. 
To show that $\hgt I(\alpha_1) \ge 3$, by Krull's principal ideal theorem, it suffices to show that $\hgt I(\alpha_1)  + (T_3) \ge 4$. Equivalently, we show that $\hgt I(\alpha_1)(S/(T_3)) \ge 3$. In fact, $[\Sym_R(I)]_i = [\Sym_R(J)]_i$ for $J = I_2 (M')$, where 
\begin{equation}
M' = \begin{bmatrix}
x & y^2 \\
y & z^2 \\
z & x^2 \\
\end{bmatrix}.
\end{equation}

So, we may assume that $S = \kk[T_0,T_1,T_2], M = M'$, and $I = J$. 
We show that $\hgt \sqrt{I(\alpha_1)} = 3$. 
Since $[l_0~ l_1] = [T_0~ T_1~ T_2] M$, and the shape of $M$, the matrices $\Phi^{l_0}_i, \Phi^{l_1}_i$ is stable under the action by the cycle $(T_0~T_1~T_2)$. In other words, if $f(T_0,T_1,T_2) \in I(\alpha_1)$, then $f(T_1,T_2,T_0), f(T_2,T_0,T_1)$ are also in $I(\alpha_1)$. Thus, to show that $\hgt \sqrt{I(\alpha_1)} = 3$, if suffices to show that $T_0 \in \sqrt{I(\alpha_1)}$.\\

Assume to the contrary $T_0$ is not in $\sqrt{I(\alpha_1)}$. Then $\sqrt{I(\alpha_1)} S_{T_0}$ is a proper ideal of $S_{T_0}$. Consider 
\begin{equation}
g_0 = l_0 \quad g_1 = A_0 y^2+ A_1 yz + A_2 z^2,
\end{equation}
where $A_0 = T_0^3 + T_1^2 T_2, A_1 = 2T_1T_2^2, A_2 = T_0^2T_1 + T_2^3$, and $\mathcal{K}' := \mathcal{K}(g_0,g_1;B)$. 
Even though $(g_0,g_1) \subsetneq (l_0,l_1)$ in $B$, these two ideals agree in $B_{T_0}$. (In the localization at $T_0$, $l_0 = xT_0 + yT_1 + zT_2 \Longleftrightarrow x = l_0 - yT_1/T_0 - zT_2/T_0$, so we may replace $x$ by a combination of $l_0,y,z,T_0,T_1,T_2$.)
In other words, the localization of $[\mathcal{K}']_i$ at $T_0$ is also a free $S_{T_0}$-resolution for $([\Sym_R(I)]_i)_{T_0}$. 
By \Cref{23case} below, we conclude that $(A_0,A_1, A_2)S_{T_0} \in \sqrt{ I(\alpha_1') S_{T_0} }$. Notice that $(A_0,A_1, A_2)S_{T_0}$ is a unit ideal in $S_{T_0}$. Hence  the Fitting lemma \cite[Cor-Def 20.4, Cor.\ 20.5]{Eis} imples that $I(\alpha_1)B_{T_0} = S_{T_0}$ as well. This is a contradiction and completes the proof of $\hgt(\alpha_1) = 3$. \\

It remains to show that $\hgt I(\alpha_1) = 2$ in the case (c). Write $r = \rank \alpha_1$. 
If $T_0^r \in I(\alpha_1)$, then modding out by $(T_1,T_3)$, we see that the image of $T_2^r$ is in the image of $I(\alpha_1)$ in $S/(T_1,T_3)$. Hence we are done by Krull's principal ideal theorem\footnote{Computations using Macaulay2 \cite{M2} for small $\nn$ shows that this case does not happen, but this is necessary for our proof.}.
Now, suppose that $T_0 \not\in \sqrt{I(\alpha_1)}$. We will show that $\hgt \sqrt{I(\alpha_1)}S_{T_0} = 2$. By the assumption, $ \sqrt{I(\alpha_1)}S_{T_0} \neq S_{T_0}$. As in the previous case, we will use 
\begin{equation}
g_0 = l_0, g_1 = A_0 y^2+ A_1 yz + A_2 z^2, g_2 = l_2,
\end{equation}
where $A_0 = T_0^3 + T_1^2 T_2, A_1 = 2T_1T_2^2, A_2 = T_0^2T_1 + T_2^3$, and $\mathcal{K}' := \mathcal{K}(g_0,g_1,g_2;B)$. 
Then the Koszul complex $\mathcal{K}'$ in degree $\nn$ is a free $S_{T_0}$-resolution of $[\Sym_R(I)]_{\nn}$. If we name $\alpha_1'$ the first differential  of $[\mathcal{K}']_{\nn}$, then $I(\alpha_1)S_{T_0} = I(\alpha_1')S_{T_0}$ since they are both the Fitting ideal of the same module. 
Therefore, we will show that $\hgt  I(\alpha_1')S_{T_0} = 2$. 
Notice that $\alpha'$ is 
\begin{equation}
\def\arraystretch{1.75}
\left[ 
\begin{array}{c;{2pt/2pt}c;{2pt/2pt}c}
\Phi^{g_0}_\nn & \Phi^{g_1}_\nn & \Phi^{g_2}_\nn 
\end{array} \right]
 = 
\left. \vphnmm \right[
\overbrace{
\underbrace{
\begin{array}{cccc}
T_0 & 0 & \cdots & 0 \\ 
T_1 & T_0 & \cdots & 0 \\
\vdots & \vdots &  \ddots & 0\\
* & * & \cdots & T_0 \\
\hdashline[1pt/1pt]~
* & * & \cdots & * \\
\vdots & \vdots &  \ddots & \vdots \\
\vdots & \vdots &  \ddots & \vdots \\
\vdots & \vdots &  \ddots & \vdots \\
\vdots & \vdots &  \ddots & \vdots \\
* & * & \cdots & * 
\end{array}
}_{ {2 + \nn - 1\choose \nn - 1}~\text{columns}}  
}^{\Phi^{g_0}_{\nn}} 
\hspace{-1em}\begin{array}{c;{2pt/2pt}c}
~&~\\ ~&~\\ ~&~\\ ~&~\\ ~&~\\ ~&~\\ ~&~\\~&~\\~&~\\
\end{array}\hspace{-1em}
\overbrace{
%
\begin{array}{ccc}
* & \cdots & *\\
* & \cdots & *\\
\vdots & \vdots & \vdots \\
\vdots & \vdots & \vdots \\
\hdashline[1pt/1pt]~ 
\vdots & \vdots & \vdots \\
\vdots & \vdots & \vdots \\
\vdots & \vdots & \vdots \\
\vdots & \vdots & \vdots \\
\vdots & \vdots & \vdots \\
* & \cdots & *
\end{array}
\hspace{-1em}\begin{array}{c;{2pt/2pt}c}
~&~\\ ~&~\\ ~&~\\ ~&~\\ ~&~\\ ~&~\\ ~&~\\~&~\\~&~\\
\end{array}\hspace{-1em}
\underbrace{
\begin{array}{cccc}
0 & 0 & \cdots & 0 \\
\vdots & \vdots & \cdots & \vdots \\
\vdots & \vdots & \cdots & \vdots \\
\vdots & \vdots & \cdots & \vdots \\
\hdashline[1pt/1pt]~ 
A_0 & 0 & \cdots & 0 \\
A_1 & A_0 & \cdots & 0 \\
A_2 & A_1 & \ddots & 0 \\
\vdots & \vdots & \vdots & \vdots \\
0 & 0 & \cdots & A_1 \\
0 & 0 & \cdots & A_2
\end{array}
}_{ {1 + \nn -2 \choose \nn-2 }~\text{columns}}
}^{\Phi^{g_1}_{\nn}}
\hspace{-1em}\begin{array}{c;{2pt/2pt}c}
~&~\\ ~&~\\ ~&~\\ ~&~\\ ~&~\\ ~&~\\ ~&~\\~&~\\~&~\\
\end{array}\hspace{-1em}
\overbrace{
\begin{array}{c}
0 \\ \vdots \\ \vdots \\ \vdots \\ \hdashline[1pt/1pt]~  T_2 \\ 0 \\ 0 \\ \vdots \\ T_1 \\ \gamma T_0 
\end{array}
}^{\Phi^{g_2}_\nn}
\left. \vphnmm \right]
\hspace{-1em}\begin{array}{c;{2pt/2pt}c}
~&~\\ ~&~\\ ~&~\\ ~&~\\ ~&~\\ ~&~\\ ~&~\\~&~\\~&~\\
\end{array}\hspace{-1em}
\overbrace{
\begin{array}{c}
x^\nn \\ x^{\nn-1}y \\ \vdots \\ x z^{\nn-1} \\ \hdashline[1pt/1pt] y^\nn \\ y^{\nn-1}z \\ \vdots \\ \vdots \\ y z^{\nn-1} \\ z^\nn.
\end{array}
}^{basis}
\end{equation} 
Since $T_0$ is invertible in $S_{T_0}$, $I(N) \subset I(\alpha') S_{T_0}$, where $N$ is the matrix obtained from the lower right corner of $\alpha'$.
\begin{equation}
N = 
\left. 
\hspace{-1em}\begin{array}{cc}
~&~\\ ~&~\\ ~&~\\ ~&~\\ ~&~\\ ~&~\\ 
\end{array}\hspace{-1em}
 \right[
\begin{array}{ccccc}
A_0 & 0 & \cdots & 0 & 0 \\
A_1 & A_0 & \cdots & 0 &  0 \\
A_2 & A_1 & \ddots & 0 & 0 \\
\vdots & \vdots & \vdots & \vdots & \vdots \\
0 & 0 & \cdots & A_2 & A_1 \\
0 & 0 & \cdots & 0 & A_2
\end{array}
\hspace{-1em}\begin{array}{c;{2pt/2pt}c}
~&~\\ ~&~\\ ~&~\\ ~&~\\ ~&~\\ ~&~\\ 
\end{array}\hspace{-1em}
\left. \begin{array}{c}
T_2 \\ 0 \\ 0 \\ \vdots \\ T_1 \\ \gamma T_0 
\end{array} \right].
\end{equation} 
We will show that after modifying $N$, $\hgt I(N) \ge 2$, so $\hgt  I(\alpha_1')S_{T_0} = 2$ by the containment above. 
(We note that $N$ and its variations $N', N'', N_T$ below will be matrices of rank $j$ and size $(j + 1) \times j$ (with different $j$).) 
Since $T_0$ is invertible in $S_{T_0}$, we may multiply the last row by $-T_2/\gamma T_0$ and add it to the first row. 
Also, we multiply the last row by $-T_1/\gamma T_0$ and add it to the second last row. 
We call this new matrix $N'$
\begin{equation}
N' = 
\left[
\begin{array}{ccccc;{2pt/2pt}c}
A_0 & 0 & \cdots & 0 & -A_2 \, T_2/\gamma T_0 & 0\\
A_1 & A_0 & \cdots & 0 &  0 & 0 \\
A_2 & A_1 & \ddots & 0 & 0 &0 \\
\vdots & \vdots & \vdots & \vdots & \vdots & \vdots \\
0 & 0 & \cdots & A_2 & A_1 - A_2T_1/\gamma T_0& 0 \\
\hdashline[1pt/1pt] 
0 & 0 & \cdots & 0 & A_2 & \gamma T_0
\end{array}
\right]
:= 
\left[
\begin{array}{c;{2pt/2pt}c}
N'' & \mathbf{0}\\
\hdashline[1pt/1pt] 
*&  \gamma T_0
\end{array}
\right].
\end{equation} 
Here $N''$ is the matrix obtained by deleting the last row and column of $N'$. 
It is straightforward to check that $I(N') = I(N'')$ as $T_0$ is a unit in $S_{T_0}$.
Consider the following matrix $N_{T}$ in $S_{T_0}[T]$ (here, T is a new variable)
\begin{equation}
N_{T} =
\left[
\begin{array}{ccccc}
A_0 T & 0 & \cdots & 0 & -A_2 \, T_2/ (\gamma T_0) \\
A_1 & A_0 & \cdots & 0 &  0 \\
A_2 & A_1 & \ddots & 0 & 0 \\
\vdots & \vdots & \vdots & \vdots & \vdots \\
0 & 0 & \cdots & A_1 & A_0 \\
0 & 0 & \cdots & A_2 & A_1 - A_2T_1/\gamma T_0 
\end{array}
\right]
\end{equation}
and let $J_T  := I(N_T)$. Since $J_1 \subseteq I(\alpha_1')S_{T_0}[T]$, $\hgt (J_T) < \infty$. 
We claim that $\hgt J_T \ge 2$. 
We will prove this by showing that $\sqrt{(J_T, T)} = (T_1,T_2,T)$ is of height $3$. 
Once we have proven the claim, then by \Cref{corGeneralHeight}, there exists a non-empty open subset $U$ of $\Spec k[T]$ such that for any (closed) point $\gamma'$ in $U$, $\hgt J_{\gamma'} \ge 2$.\\

First of all, we have $(J_T, T) \subset (A_1, A_2, T)$ since the ideal generated by $A_1,A_2,T$  contains all the entries of the first column. 
Furthermore, if a prime $\p$ contains $(A_1,A_2,T) = (2T_1T_2^2, T_0^2T_1 + T_2^3, T)$, then $T_1,T_2,T$ are in $\p$. 
(Recall that we are in $S_{T_0}$, so $T_0$ is a unit.) 
Thus, we conclude that $\sqrt{ (A_1, A_2, T) }= (T_1,T_2,T)$. 
Consider the minor obtained by deleting the second row of $N_T$ and then putting the first row last 
\begin{equation}\label{labA2}
 \left|
\begin{array}{ccccccc}
A_2 & A_1 & 0 & 0 &\cdots & 0  \\
0 & A_2 & A_1 & A_0 & \cdots & 0  \\
\vdots & \vdots & \vdots & \vdots  & \vdots & \vdots \\
0 & 0 & \cdots & A_2 & A_1 & A_0 \\
0 & 0 & \cdots & 0 & A_2 &  A_1 - A_2T_1/\gamma T_0  \\
A_0 T & 0 & \cdots & 0 & 0 & -A_2 \, T_2/ (\gamma T_0) \\
\end{array}
\right|.
\end{equation}
From this minor, we conclude that $A_2^{\rank N_T}T_2/ (\gamma T_0)$ is in $(J_T,T)$. 
Therefore, if $\p$ is a prime ideal containing $(J_T,T)$, then it contains $A_2$ or $T_2$. 
If $A_2 \in \p$, then from the minor obtained by deleting the first row of $N_T$, we have $A_1 \in \p$. 
In this case, we have $\sqrt{(J_T,T)} = \sqrt{(A_1,A_2,T)} = (T_1,T_2,T) \subset \p$, so $(T_1,T_2,T) = \p$ since $(T_1,T_2,T)$ is a prime of height $3$. 
We show that we have the same conclusion in the case where $T_2 \in \p$. 
Assume $T_2 \in \p$, then since $A_1 = 2T_1T_2^2$, $A_1$ is in $\p$. 
Hence from the minor obtained by deleting the first row of $N_T$ and from the fact that $T, T_2, A_1 \in \p$, the determinant
\begin{equation}
\left| \begin{array}{cccccc}
0 & A_0 & 0 & 0&\cdots & 0 \\
A_2 & 0 & A_0 &0 &\cdots & 0 \\
0 & A_2 & 0 & 0&\cdots & 0 \\
\vdots & \vdots & \vdots&\vdots & \vdots & \vdots\\
0 & 0 & \cdots &A_2& 0& A_0\\
0 & 0 & \cdots &0& A_2 &  - A_2T_1/\gamma T_0
\end{array} \right|
\end{equation}
is in $\p$. 
By \Cref{lemOffDiag}, $A_2^{\nn/2}A_0^{\nn/2}$ is in $\p$ if $\nn$ is even, and $A_2^{(\nn-1)/2}A_0^{(\nn-1)/2}(- A_2T_1/\gamma T_0)$ is in $\p$ if $\nn$ is odd.
Since $T_0^3 \in (A_0, T_2) \subset (A_0, \p)$ is a unit ideal, $A_0 \not\in \p$. 
Therefore, in both cases we have $A_2$ or $T_1$ is in $\p$. 
Since $T_2$ is in $\p$ and $T_0$ is a unit, $A_2 = T_0^2T_1 +T_2^3$ is in $\p$ if and only if $T_1$ is in $\p$. 
Hence in both cases we have $A_2 \in \p$.  
We already showed that if $A_2 \in \p$, then $\p = (T_1,T_2,T)$ (in the paragraph before \cref{labA2}).
Hence we have $\sqrt{(J_T,T)} = \sqrt{(A_2,A_2,T)} = (T_1,T_2,T)$, and this proves the claim. \\

By \Cref{corGeneralHeight}, there exists a non zero $\gamma'$ in $\kk$ such that 
$\hgt N_{\gamma'} = 2$. 
But the maximal minors of the following matrices 
\begin{equation}
N_{\gamma'} =
\left[
\begin{array}{ccccc}
A_0 \gamma' & 0 & \cdots & 0 & -A_2 \, T_2/ (\gamma T_0) \\
A_1 & A_0 & \cdots & 0 &  0 \\
A_2 & A_1 & \ddots & 0 & 0 \\
\vdots & \vdots & \vdots & \vdots & \vdots \\
0 & 0 & \cdots & A_2 & A_1 
\end{array}
\right],
\quad
N'' = 
\left[
\begin{array}{ccccc}
A_0 & 0 & \cdots & 0 & -A_2 \, T_2/ (\gamma \gamma' T_0) \\
A_1 & A_0 & \cdots & 0 &  0 \\
A_2 & A_1 & \ddots & 0 & 0 \\
\vdots & \vdots & \vdots & \vdots & \vdots \\
0 & 0 & \cdots & A_2 & A_1 
\end{array}
\right]
\end{equation}
generate the same ideal. 
We replace $\gamma$ by $\gamma \gamma'$. This does not change our argument since the previous calculation and proof depended only on the fact that $\gamma$ was not zero. 
Therefore, with new $\gamma$, we obtain the desired height for $I (N)$, and this completes the proof.
\end{proof}
\end{example}

\begin{lemma}\label{23case}
Let $S = \kk[T_0,T_1,T_2]$, where $\kk$ is a field, $B = S[x,y,z], \m = (x,y,z)B,  g_0 = T_0x + T_1 y+ T_2 z, g_1= A_0 y^2+ A_1 yz + A_2 z^2$ with $A_i \in S$. If  $g_0,g_1$ form an $B$-regular sequence, then for any $i \ge 2$, the ith graded component (with respect to $x,y,z$ degree) of $B/(g_0,g_1)$  has a representation $\alpha$ of the form 
\begin{equation}
\def\arraystretch{1.8}
\left[ 
\begin{array}{c;{2pt/2pt}c}
\Phi^{g_0}_{i} & \Phi^{g_1}_{i}
\end{array} \right]
 = 
\left. \vphnm \right[
\overbrace{
\underbrace{
\begin{array}{cccc}
T_0 & 0 & \cdots & 0 \\ 
T_1 & T_0 & \cdots & 0 \\
\vdots & \vdots &  \ddots & 0\\
* & * & \cdots & T_0 \\
\hdashline[1pt/1pt]~
* & * & \cdots & * \\
\vdots & \vdots &  \ddots & \vdots \\
\vdots & \vdots &  \ddots & \vdots \\
\vdots & \vdots &  \ddots & \vdots \\
* & * & \cdots & * 
\end{array}
}_{ {2 + i - 1\choose i - 1}~\text{columns}}  
}^{\Phi^{g_0}_{i}} 
\hspace{-1em}\begin{array}{c;{2pt/2pt}c}
~&~\\ ~&~\\ ~&~\\ ~&~\\ ~&~\\ ~&~\\ ~&~\\~&~\\
\end{array}\hspace{-1em}
\overbrace{
%
\begin{array}{ccc}
* & \cdots & *\\
* & \cdots & *\\
\vdots & \vdots & \vdots \\
\vdots & \vdots & \vdots \\
\hdashline[1pt/1pt]~ 
\vdots & \vdots & \vdots \\
\vdots & \vdots & \vdots \\
\vdots & \vdots & \vdots \\
\vdots & \vdots & \vdots \\
* & \cdots & *
\end{array}
\hspace{-1em}\begin{array}{c;{2pt/2pt}c}
~&~\\ ~&~\\ ~&~\\ ~&~\\ ~&~\\ ~&~\\ ~&~\\~&~\\
\end{array}\hspace{-1em}
\underbrace{
\begin{array}{cccc}
0 & 0 & \cdots & 0 \\
\vdots & \vdots & \cdots & \vdots \\
\vdots & \vdots & \cdots & \vdots \\
\vdots & \vdots & \cdots & \vdots \\
\hdashline[1pt/1pt]~ 
A_0 & 0 & \cdots & 0 \\
A_1 & A_0 & \cdots & 0 \\
A_2 & A_1 & \ddots & 0 \\
\vdots & \vdots & \vdots & \vdots \\
0 & 0 & \cdots & A_2
\end{array}
}_{ {1 + i -2 \choose i-2 }~\text{columns}}
}^{\Phi^{g_1}_{i}}
\left. \vphnm \right]
\hspace{-1em}\begin{array}{c;{2pt/2pt}c}
~&~\\ ~&~\\ ~&~\\ ~&~\\ ~&~\\ ~&~\\ ~&~\\~&~\\
\end{array}\hspace{-1em}
\overbrace{
\begin{array}{c}
x^i \\ x^{i-1}y \\ \vdots \\ x z^{i-1} \\ \hdashline[1pt/1pt] y^i \\ y^{i-1}z \\ \vdots \\ \vdots \\ z^i 
\end{array}
}^{basis}
\end{equation} 
Furthermore, we have $\rk \alpha = {2+i-1 \choose 2} + {i-1}$ and $T_0^{2+i-1 \choose 2}(A_0,A_1,A_2)^{i-1} \subset I(\alpha)$.
\end{lemma}

\begin{proof}
We use the polynomial grading on $B$ and write $M := B/(g_0,g_1)$.
Since $g_0,g_1$ form a regular sequence, the Koszul complex $\mathcal{K}:=\mathcal{K}(g_0,g_1;B)$ is a grade $B$-resolution of $M$. By \Cref{pieceRes}, $\mathcal{K}_i$ is a free $S$-resolution of $M_i$. That is 
\begin{equation}
0 \to S^{2+i-3 \choose 2} \stackrel{\alpha_2}\to S^{2+ i-1 \choose 2} \oplus S^{2+i-2 \choose2}  \stackrel{\alpha}{\to} S^{2 + i  \choose 2} \to M_i \to 0
\end{equation}
is exact. Thus, 
\begin{align}
\rk M_i &= {2 + i  \choose 2} + {2+i-3 \choose 2} - {2+ i-1 \choose 2} - {2+i-2 \choose 2} \\
&= \left[{2 + i  \choose 2} - {2+ i-1 \choose 2}\right] - \left[{2+i-2 \choose2} - {2+i-3 \choose 2}\right] \\
&= {2 + i - 1 \choose 1} - {2 + i -3 \choose 1} = 2,
\end{align}

and $\rk \alpha = {2 + i  \choose 2} - 2$.
The presentation follows by using the Lex order on the monomials (in terms of $x,y,z$) of $B$. Since 
\begin{align}
{2+i-1 \choose 2} + \dim_{\kk} (y,z)^{i-2} &=(i+1)i/2 + (i -1) = (i^2 + 3i)/2 - 1  = (i^2+3i+2)/2 - 2 \\
&= (i+2)(i+1)/2 - 2={2+i \choose 2}-2 = \rank \alpha,
\end{align}
 we see that the $T_0^{2+i-1 \choose 2}I(N) \subset I(\alpha)$, where 
\begin{equation}
N = \left[ \begin{array}{cccc}
A_0 & 0 & \cdots & 0 \\
A_1 & A_0 & \cdots & 0 \\
A_2 & A_1 & \cdots & 0 \\
\vdots & \vdots & \vdots & \vdots \\
0 & 0 & \cdots & A_2
\end{array} \right].
\end{equation}
By \Cref{lemStaircase}, $I(N) = (A_0,A_1,A_2)^{i-1}$, and this proves the statement. 
\end{proof}

\begin{lemma}\label{lemOffDiag}
Let $x,y,z$ be variables over $\mathbb{Z}$. Then we have
\begin{equation}
\left| \begin{array}{cccccc}
0 & x & 0 & 0&\cdots & 0 \\
y & 0 & x &0 &\cdots & 0 \\
0 & y & 0 & x&\cdots & 0 \\
\vdots & \vdots & \vdots&\vdots & \vdots & \vdots\\
0 & 0 & \cdots &y& 0& x\\
0 & 0 & \cdots &0& y& 0
\end{array} \right|
= \begin{cases} 
(-1)^{\frac n2}x^{\frac n2} y^{\frac n2}, & if~n~is~even \\
0, & if~n~is~odd
\end{cases}
~~~~~and~
\left| \begin{array}{cccccc}
0 & x & 0 & 0&\cdots & 0 \\
y & 0 & x &0 &\cdots & 0 \\
0 & y & 0 & x&\cdots & 0 \\
\vdots & \vdots & \vdots&\vdots & \vdots & \vdots\\
0 & 0 & \cdots &y& 0& x\\
0 & 0 & \cdots &0& y& z
\end{array} \right|
= \begin{cases} 
(-1)^{\frac n2} x^{\frac n2} y^{\frac n2}, & if~n~is~even \\
(-1)^{\frac{n-1}2}x^{\frac{n-1}2} y^{\frac{n-1}2} z, & if~n~is~odd.
\end{cases}
\end{equation}
\end{lemma}
\begin{proof}
For the first determinant, one can use induction on the size of the matrix, and for the second determinant, one can use the Laplace expansion along the last row (or column) and use the result of the first determinant.
\end{proof}

\begin{remark}
\item The non-empty open subset for the conditions in \Cref{mainThm} is a proper subset of the parameter space. The following matrix provides an example where $I:= I_3(M)$ satisfies condition (1) in the proof of \Cref{mainThm}, but not condition (2). \begin{equation}
M = 
\begin{bmatrix}
x	&0	&x^3\\
y	&x^2	&0\\
0	&y^2	&z^3\\
0	&z^2	&0
\end{bmatrix}.
\end{equation}
It is interesting that the rational map induced by $I$ is birational to its image. 
\end{remark}

We will end the paper with a couple of questions.
\begin{question}\label{lastQues}
\begin{enumerate}
\item 
Let $U$ be the non-empty open subset corresponds to the general condition in \Cref{mainThm}. Can we determine the closed set which is the complement of $U$? In the case where $R = \kk[x_1,x_2]$, Kustin, Polini, Ulrich were able to connect this condition to a geometric condition (See \cite[Lemma 2.10]{KPU17}).\\

\item 
In the general case where $R = \kk[x_1,\dots, x_n]$, does the following  presentation matrix (or its variations) of $I = I_n(\varphi)$
\begin{equation}
\varphi = \begin{bmatrix}
x_1^{d_1} &  x_1^{d_2} & \cdots & x_1^{d_n} \\
x_2^{d_1} & x_2^{d_2} & \cdots & x_2^{d_n} \\
\vdots & \vdots & \ddots &  \vdots \\
x_n^{d_1} & x_n^{d_n} & \cdots & x_n^{d_n} \\
g_1 & \cdots & \cdots & g_n 
\end{bmatrix},
\end{equation}
where $g_i$ are symmetric polynomials of degree $d_i$, provide an example of defining a similar general property described in \Cref{mainThm}?\\

Note that the corresponding presentation matrices for the graded pieces of $\Sym_R(I)$ are invariant under the action of the cyclic group generated by the cycle $(T_0, \dots, T_{n-1})$. We hope an expert can utilize this fact and provide an alternative way to verify \Cref{bigExample}. 
Recall that the more general $\varphi$ becomes, the more challenging to verify conditions. 
\end{enumerate}
\end{question}

\end{document}